\pdfoutput=1
\documentclass[11pt,twoside]{article}
\usepackage{amsfonts}
\usepackage{color}
\usepackage{fancyhdr}
\usepackage{titlesec}
\usepackage{cite}
\usepackage{ifthen}
\usepackage{amssymb}
\usepackage{fancyhdr}
\usepackage{titlesec}
\usepackage{pifont}
\usepackage{setspace}
\usepackage{indentfirst}
\usepackage{amsmath,amssymb,amscd,amsthm,mathrsfs}
\input amssym.def

\newboolean{first}
\setboolean{first}{true}
\renewcommand{\headrulewidth}{0pt}

\newfont{\aaa}{cmb10 at 19pt}
\newfont{\bbb}{cmb10 at 11pt}
\newtheorem{lemma}{Lemma}[section]
\newtheorem{theorem}{Theorem}[section]
\newtheorem{definition}{Definition}[section]
\newtheorem{rem}{Remark}[section]
\newtheorem{corollary}{Corollary}[section]
\newtheorem{proposition}{Proposition}[section]
\newcommand{\beq}{\begin{equation}}
\newcommand{\eeq}{\end{equation}}
\newcommand{\bey}{\begin{eqnarray}}
\newcommand{\eey}{\end{eqnarray}}
\newcommand{\beyy}{\begin{eqnarray*}}
\newcommand{\eeyy}{\end{eqnarray*}}

\numberwithin{equation}{section}
\makeatletter
\def\@evenfoot{}
\def\@oddfoot{}
\makeatother
\makeatletter
\newcommand{\Rmnum}[1]{\expandafter\@slowromancap\romannumeral #1@}
\makeatother

\begin{document}
\thispagestyle{empty} \thispagestyle{fancy} {
\fancyhead[RO,LE]{\scriptsize \bf 
} \fancyfoot[CE,CO]{}}
\renewcommand{\headrulewidth}{0pt}

\begin{center}
{\bf \LARGE An Extension of De Giorgi Class and Applications}

\vspace{3mm}

{\small    \textsc {Hongya GAO}$^1$ \quad \textsc{Aiping ZHANG}$^1$  \quad   \textsc {Siyu GAO}$^2$\footnote{Corresponding author, email: siyugao@my.unt.edu.}

\vspace{2mm}

{\small 1. College of Mathematics and Information Science, Hebei University, Baoding, 071002, China\\
2. Department of Mathematics, University of North Texas, Denton, Texas 76203, USA }}
\end{center}

\vspace{4mm}

\begin{center}
\begin{minipage}{135mm}
{\bf \small Abstract}. {\small We present an extension of the classical De Giorgi class, and then we show that functions in this new class are locally bounded and locally H\"older continuous. Some applications are given. As a first application, we give a regularity result for local minimizers $u:\Omega \subset \mathbb R^4 \rightarrow \mathbb R^4$ of a special class of polyconvex functionals  with splitting form in four dimensional Euclidean spaces. Under some structural conditions on the energy density, we prove that each component $u^\alpha$ of the local minimizer $u$ belongs to the generalized De Giorgi class, then one can derive that it is locally bounded and locally H\"older continuous. Our result can be applied to polyconvex integrals whose prototype is
$$
\int_\Omega \Big(\sum_{\alpha =1}^4 |Du^\alpha|^p + \sum _{\beta =1}^6 |({\rm adj}_2 Du )^\beta | ^q  +\sum_{\gamma =1}^4 |({\rm adj}_3 Du )^\gamma | ^r  +|\det Du|^s \Big ) \mathrm {d}x
$$
with suitable $p,q,r,s\ge 1$. As a second application, we consider a degenerate linear elliptic equation of the form
$$
-\mbox {div} (a(x)\nabla u)=-\mbox {div}F,
$$
with $0<a(x) \le \beta <+\infty$. We prove, by virtue of the generalized De Giorgi class, that any weak solution is locally bounded and locally H\"older continuous provided that $\frac 1 {a(x)}$ and $F(x)$ belong to some suitable locally integrable function spaces. As a third application, we show that our theorem can be applied in dealing with regularity issues of elliptic equations with non-standard grow conditions. As a fourth application we treat with quasilinear elliptic systems. Under suitable assumptions on the coefficients, we show that any of its weak solutions is locally bounded and locally H\"older continuous.

\vspace{2mm}


}
\end{minipage}
\end{center}

\thispagestyle{fancyplain} \fancyhead{}
\fancyhead[L]{\textit{}\\
} \fancyfoot{}

\vskip 5mm

\section{The classical De Giorgi class.}
It is well-known (see, for example, Chapter 7 in \cite{Giusti}) that the quasi-minima and $\omega$-minima of regular functionals of the calculus of variations are H\"older continuous functions. The main result is a version of the fundamental theorem of De Giorgi \cite{De-Giorgi} and Nash \cite{Nash} concerning the regularity of solutions of linear elliptic equations with discontinuous coefficients, a result that was later generalized among others by Lady\v{z}enskaya and Ural'ceva \cite{LU} to bounded solutions to nonlinear elliptic equations. De Giorgi's method relies on the fact that functions lies in the De Giorgi class are locally H\"older continuous.


\vspace{2mm}

We first recall the definition of De Giorgi class, see, for example, De Giorgi \cite{De-Giorgi}.

\begin{definition}\label{definition1} Let $\Omega\subset \mathbb R^n$, $n\ge 2$, be a bounded domain.
We say that $u\in W_{loc} ^{1,p} (\Omega)$, $p\le n$,  belongs to the De Giorgi class $DG_p^+(\Omega, p, y, y_*$, $\varepsilon, \kappa _0)$, $p>1$, $y$ and $\varepsilon>0$, $y_*$ and $\kappa_0 \ge 0$ if
\begin{equation}\label{DG+}
\int _{Q_{\sigma \rho} (x_0)} |D(u-k)_+ |^p dx \le y\int_{Q_\rho (x_0)}\left( \frac {u-k} {(1-\sigma )\rho} \right)_+ ^pdx +y_* \big| \{u>k\} \cap  \{Q_\rho (x_0)\} \big| ^{1-\frac p n +\varepsilon},
\end{equation}
for all $k\ge \kappa _0$, $\sigma \in (0,1)$, and all pairs of concentric cubes $Q_{\sigma \rho} (x_0) \subset Q_\rho (x_0) \subset \Omega$ centered at $x_0$,
where
$$
(u-k)_+ =(u-k)\wedge 0 = \max \{u-k,0\},
$$
and $Q_\rho (x_0)$, $\rho >0$, $x_0\in \Omega$, be the cube with side length $2 \rho$, with sides parallel to the coordinate axis and centered at $x_0$.
\end{definition}

In the sequel we denote $Q_\rho (x_0) =Q_\rho$ provided that no confusion may arise.

If we introduce
$$
A_{k,\rho}=\{x\in Q_\rho: u(x)>k\},
$$
then (\ref{DG+}) is equivalent to
\begin{equation}\label{DG+2}
\int _{A_{k,\sigma \rho}} |Du|^p dx \le y\int_{A_{k,\rho}} \left( \frac {u-k} {(1-\sigma)\rho} \right)^pdx +y_* \big| A_{k,\rho} \big| ^{1-\frac p n +\varepsilon}.
\end{equation}

One can define similarly $DG_p^- (\Omega, p, y, y_*,\varepsilon, \kappa_0)$ to be the class of functions $u$ such that $-u \in DG_p^+(\Omega, p, y, y_*,\varepsilon, \kappa_0)$. More explicitly, they are the functions in $W_{loc} ^{1,p} (\Omega)$ such that for all $k\le -\kappa _0$, all $\sigma \in (0,1)$, and all pairs of concentric cubes $Q_{\sigma \rho}\subset Q_\rho \subset \Omega$,
\begin{equation}\label{DG-}
\int _{B_{k,\sigma \rho}} |Du|^p dx \le y\int_{B_{k,\rho}} \left( \frac {k-u} { (1-\sigma )\rho} \right)^pdx +y_* \big| B_{k,\rho} \big| ^{1-\frac p n +\varepsilon},
\end{equation}
where
$$
B_{k,\rho} =\{x\in Q_\rho: u(x)<k\}.
$$

It is clear that if a function $u$ satisfies (\ref{DG+2}) or (\ref{DG-}) with some $\varepsilon >0$, it will verify them with any positive $\varepsilon ' \le \varepsilon$. Consequently, we shall always assume $\varepsilon \le \frac p n$.

We indicate by $DG_p(\Omega, p, y, y_*,\varepsilon, \kappa_0)$ the class of functions belonging both to $DG_p^+(\Omega, p,$ $y, y_*,\varepsilon, \kappa_0)$ and $DG_p^-(\Omega, p, y, y_*,\varepsilon, \kappa_0)$:
$$
DG_p (\Omega, p, y, y_*,\varepsilon, \kappa_0)= DG_p^+(\Omega, p, y, y_*,\varepsilon, \kappa_0) \cap DG_p^-(\Omega, p, y, y_*,\varepsilon, \kappa_0).
$$

We notice that (\ref{DG+2}) and (\ref{DG-}) are Caccioppoli type inequalities on super-/sub-level sets.
A rather surprising characteristic of De Giorgi class is that (\ref{DG+2}) and (\ref{DG-}) contain particularly much information deriving from the minimum properties of the function $u$, at least for what concerns its local boundedness and local H\"older continuity. The following proposition can be found, for example, in \cite{DiBenedetto} Theorems 2.1 and 3.1 and \cite{Giusti} Chapter 7.

\begin{proposition}\label{proposition 3}
Let $u\in DG_p (\Omega, p, y,y_*,\varepsilon ,\kappa_0)$ and $\sigma  \in (0,1)$. There exist a constant $C>1$ depending only upon the data and independent of $u$, such that for every pair of cubes $Q_{\sigma  \rho}\subset Q_\rho \subset \subset \Omega$,
$$
\|u\| _{L^\infty (Q_{\sigma  \rho})} \le \max \left\{y_*\rho^{n\varepsilon  }; \frac {C}{(1-\sigma ) ^{1/\varepsilon  } } \left(\frac 1 {|Q_\rho|} \int_{Q_\rho} |u|^p \mathrm {d}x \right) ^{\frac 1 p }  \right\},
$$
moreover, there exists $\tilde \alpha \in (0,1)$ depending only upon the data and independent of $u$, such that
$$
{\rm osc} (u,Q_\rho) \le C \max \left\{y_*\rho^{n\varepsilon }; \left(\frac \rho R\right) ^{\tilde \alpha} {\rm osc} (u,Q_R)  \right\},
$$
where
$$
{\rm osc} (u,Q_\rho) ={\rm esssup }_{Q_\rho} u -{\rm essinf} _{Q_\rho}u
$$
is the oscillation of $u$ over $Q_\rho$. Therefore, $u\in C_{loc} ^{0,\tilde \alpha_0} (\Omega)$ with $\tilde \alpha _0 =\tilde \alpha \wedge (n\varepsilon )$.
\end{proposition}

The above proposition illustrates that functions in the De Giorgi class are locally bounded and locally H\"older continuous in $\Omega$. This result was first proved by De Giorgi in his famous paper \cite{De-Giorgi}, which opened the way to the regularity of solutions of elliptic equations with bounded measurable coefficients, and for minima of regular functionals in the Calculus of Variations. De Giorgi's theorem was later generalized by various authors, so as to cover the most general solutions of nonlinear equations in divergence form. We note in particular the papers by Stampacchia \cite{Stampacchia1,Stampacchia2,Stampacchia3} and the book by Lady\v{z}enskaya and Ural'ceva \cite{LU}.
Almost at the same time, a different proof of the regularity of solutions to parabolic and elliptic equations was given by Nash \cite{Nash}.
Slightly later, Moser \cite{moser} proved Harnack's inequality, thus extending to solutions of linear equations in divergence form a classical result for harmonic functions. Starting from Harnack's inequality, Moser gave a new proof of the H\"older continuity of solutions of elliptic equations.
The extension of the method of De Giorgi to minima (and quasi-minima) of functionals, independently of their Euler equation, was made by Giaquinta and Giusti \cite{GG}, after Frehse \cite{Frehse} has studied a particular case, under rather restrictive hypothesis.
Harnack's inequality was proved by Di Benedetto and Trudinger \cite{BT} for functions in De Giorgi class, and hence for quasi-minima of integral functionals
$$
{\cal F} (u,\Omega) =\int_\Omega f(x,u(x),Du(x))dx .
$$
For some recent developments of De Giorgi class and its applications, we refer the reader to \cite{SkrypnikandVoitovych,SkrypnikandVoitovych2,Cozzi,Castro,Zacher}.

\section{An extension.}
In this section, we shall give an extension of the classical De Giorgi class, and prove that it is closely related the regularity properties of the function $u$, including local boundedness and local H\"older continuity properties.

In the following, for $1< p\le n$, we shall use the symbol $p^*$ which is defined as: $p^*=\frac {np}{n-p}$ if $p<n$, and $p^*=$ any $\nu >p$ for $p=n$.

We give the following

\begin{definition}\label{definition 2}
We say that $u\in W_{loc} ^{1,p} (\Omega)$ belongs to the generalized De Giorgi class $GDG_{p}^+(\Omega, p, Q, y, y_*,\varepsilon ,\kappa_0)$, $1< p\le n$, $p\le Q<p^*$, $y$ and $\varepsilon>0$, $y_*$ and $\kappa _0\ge 0$,  if
\begin{equation}\label{eDG+}
\int _{A_{k,\sigma \rho}} |Du|^p dx \le y\int_{A_{k,\rho}} \left( \frac {u-k} { (1-\sigma)\rho} \right)^Q dx +y_* \big| A_{k,\rho} \big| ^{1-\frac p n +\varepsilon},
\end{equation}
for all $k\ge \kappa_0 $, $\sigma \in (0,1)$, and all pairs of concentric cubes $Q_{\sigma \rho} (x_0) \subset Q_\rho (x_0) \subset \Omega$ centered at $x_0$.
\end{definition}

We can define similarly $GDG_{p}^- (\Omega, p, Q,y, y_*,\varepsilon ,\kappa_0)$ to be the class of functions $u$ such that $-u \in GDG_p^+ (\Omega, p, Q,y, y_*,\varepsilon ,\kappa_0)$. More explicitly, they are the functions in $W_{loc} ^{1,p} (\Omega)$ such that for all $k\le -\kappa _0$, all $\sigma \in (0,1)$, and all pairs of concentric cubes $Q_{\sigma \rho}\subset Q_\rho \subset \Omega$,
\begin{equation}\label{eDG-}
\int _{B_{k,\sigma \rho}} |Du|^p dx \le y\int_{B_{k,\rho}} \left( \frac {k-u} {(1-\sigma )\rho} \right)^Q dx +y_* \big| B_{k,\rho} \big| ^{1-\frac p n +\varepsilon}.
\end{equation}

We shall indicate by $GDG_{p}(\Omega, p, Q,y, y_*,\varepsilon ,\kappa_0)$ the class of functions belonging both to $GDG_{p}^+$ $(\Omega, p, Q,y, y_*,\varepsilon ,\kappa_0)$ and $GDG_{p}^-(\Omega, p, Q,y, y_*,\varepsilon ,\kappa_0)$:
\begin{equation}\label{definition for GDG-1}
\begin{array}{llll}
&\displaystyle GDG_{p} (\Omega, p, Q,y, y_*,\varepsilon ,\kappa_0)\\
= &\displaystyle GDG_{p}^+(\Omega, p, Q,y, y_*,\varepsilon ,\kappa_0) \cap GDG_{p}^-(\Omega, p, Q,y, y_*,\varepsilon ,\kappa_0).
\end{array}
\end{equation}

It is clear that $GDG_{p}(\Omega, p, p,y, y_*,\varepsilon ,\kappa_0) =DG_p(\Omega, p,y, y_*,\varepsilon ,\kappa_0)$. If no confusion may arise, $GDG_{p}(\Omega, p, Q,y, y_*,\varepsilon ,\kappa_0)$ will be abbreviated as $GDG_p$.


We remark that the difference between (\ref{DG+2}) and (\ref{eDG+}) (similarly (\ref{DG-}) and (\ref{eDG-})) is that the power $p$ in the first integral in the right hand side is replaced by a power $Q$ which is greater than or equals to $p$ but smaller than $p^*$. It is obvious that (\ref{eDG+}) is weaker than (\ref{DG+2}) (similarly (\ref{eDG-}) is slightly weaker than (\ref{DG-})).


\vspace{2mm}

The main result of this paper is the following

\begin{theorem}\label{theorem 21}
Let $u \in GDG_{p} (\Omega, p, Q,y, y_*,\varepsilon,\kappa_0)$ for $1<p\le n$ and some $Q \in [p, p^*)$, then $u$ is locally bounded and locally H\"older continuous in $\Omega$.
\end{theorem}

The above theorem illustrates that, functions in the generalized De Giorgi class $GDG_{p}$ $(\Omega, p, Q,y, y_*,\varepsilon, \kappa _0)$ has also remarkable regularity properties as De Giorgi class does, in particular, they are locally bounded and locally H\"older continuous.

\vspace{2mm}

We notice that in the above Theorem \ref{theorem 21} we have restricted ourselves to the case $1<p\le n$, since if $p>n$, then any function $u\in W_{loc} ^{1,p} (\Omega)$ is trivially in $C_{loc} ^{0,\tilde \alpha } (\Omega)$ for some $0<\tilde \alpha <1$ by Sobolev Imbedding Theorem, i.e., it is automatically a H\"older continuous function.

\vspace{2mm}

We should mention that in (\ref{eDG+}) and (\ref{eDG-}) we have restricted ourselves to $p \le Q<p^*$. One may wonder if $u\in GDG_{p} (\Omega, p, Q,y, y_*,\varepsilon, \kappa _0)$ is also locally bounded and locally H\"older continuous for the case $Q=p^*$. Despite some efforts, we can not prove this result. Fortunately, the case $Q<p^*$ is enough for our purposes to derive some regularity results of minimizers of some variational integrals,  as well as weak solutions of some elliptic equations and systems, see Sections 3, 4, 5 and 6.

\vspace{2mm}

We remark that, in the proof of Theorem \ref{theorem 21} we borrow some ideas from \cite{De-Giorgi,Giusti}.

\vspace{2mm}

In the following we shall denote by $c(\cdots)$ a constant depending only on the quantities, whose value may vary from line to line.

\vspace{2mm}

We divide the proof of Theorem \ref{theorem 21} into several lemmas.

\begin{lemma}\label{locally bounded}
Let $u\in GDG_{p} (\Omega, p, Q, y, y_*,\varepsilon,\kappa _0)$ for $1<p\leq n$ and some $Q\in [p,p^*) $. Then $u$ is locally bounded in $\Omega$.
\end{lemma}

\begin{proof} We notice that if $u\in GDG_{p}^+(\Omega, p, Q, y, y_*,\varepsilon, \kappa_0)$, then Young inequality with exponents $\frac {p^*} Q$ and $\frac {p^*}{p^*-Q}$ allows us to estimate
$$
\begin{array}{llll}
&& \displaystyle \int _{A_{k,\sigma \rho}} |Du|^p dx \\
&\le &\displaystyle  y\int_{A_{k,\rho}} \left( \frac {u-k} {(1-\sigma)\rho} \right)^Q dx +y_* \big| A_{k,\rho} \big| ^{1-\frac p n +\varepsilon}\\
&\le &\displaystyle c\left( \int_{A_{k,\rho}} \left( \frac {u-k} {(1-\sigma)\rho} \right)^{p^*} dx+|A_{k,\rho}| + \big| A_{k,\rho} \big| ^{1-\frac p n +\varepsilon} \right)\\
&\le &\displaystyle c \left( \int_{A_{k,\rho}} \left( \frac {u-k} {(1-\sigma)\rho} \right)^{p^*} dx+ \big| A_{k,\rho} \big| ^{1-\frac p n +\varepsilon} \right) ,
\end{array}
$$
here we have used the facts $\varepsilon \le \frac p n $ and $|A_{k,\rho}|\le |\Omega|$, which imply
$$
|A_{k,\rho}| = |A_{k,\rho}| ^{\frac p n -\varepsilon} |A_{k,\rho}| ^{1-\frac p n +\varepsilon} \le |\Omega| ^{\frac p n -\varepsilon} |A_{k,\rho}| ^{1-\frac p n +\varepsilon}.
$$
We shall use this fact repeatedly in the sequel. Thus $u$ satisfies the inequality (2.6) in \cite{Cupini-Leonetti-Mascolo} (the only difference between the above inequality and (2.6) in \cite{Cupini-Leonetti-Mascolo} is that cubes in place of balls. In the following we will not distinguish between cubes and balls). It has been proved in \cite{Cupini-Leonetti-Mascolo}, by using De Giorgi's iteration method, that any function satisfying the above inequality is locally bounded from above.

Analogously, if $u\in GDG_{p}^- (\Omega, p, Q, y, y_*,\varepsilon,\kappa _0) $, then it is locally bounded from below.

The locally boundedness result follows since (\ref{definition for GDG-1}).
\end{proof}

The following is a technical lemma, which can be found, for example, in  Lemma 7.1 in \cite{Giusti}.

\begin{lemma}\label{lemma 1.81}
Let $ \alpha>0$  and let $ \left\{x_{i}\right\}$  be a sequence of real positive numbers such that
$$
x_{i+1} \leq C B^{i} x_{i}^{1+\alpha}
$$
with $ C>0$ and $ B>1$. If
$$
x_{0} \leq C^{-\frac{1}{\alpha}} B^{-\frac{1}{\alpha^{2}}},
$$
then we have
$$
x_{i} \leq B^{-\frac{i}{\alpha}} x_{0},
$$
and hence in particular
$$
\lim _{i \rightarrow \infty} x_{i}=0.
$$
\end{lemma}

We next prove

\begin{lemma}\label{proposition 2}
Let $u(x)$ be a locally bounded function, verifying there exists $c>0$, such that for every $0< r <\rho\leq R$, $k\ge \kappa_0\ge 0$ and $p\le Q <p^*$,
\begin{equation}\label{1.6}
\int_{A_{k, r}}|D u|^{p} dx \leq  c \int_{A_{k, \rho}} \Big( \frac{u-k}{\rho-r} \Big)^{Q} d x +c |A_{k, \rho}|^{1-\frac{p}{n}+\varepsilon},
\end{equation}
then
\begin{equation}\label{3.17}
\begin{aligned}
\sup _{Q_{\frac{R}{2}}} (u-\kappa _0) \leq c \left(\frac{1}{R^{\frac{(np-Q)n}{p(n-1)}}} \int_{A_{\kappa _{0}, R}}\left(u-\kappa _{0}\right)^{Q} d x\right)^{\frac{1}{Q}}\left(\frac{\left|A_{\kappa _{0}, R}\right|}{R^{ \frac{(np-Q)n}{p(n-1)}}}\right)^{\frac{\gamma}{Q}} +c R^{\tau},
\end{aligned}
\end{equation}
where  $\gamma$ is the positive solution of the equation
$$
\gamma^{2}+\gamma=\varepsilon,
$$
and $\tau$ is a small positive number depending on $n,p,Q$.
\end{lemma}

\begin{proof}
We can suppose $\kappa_{0}=0$,  and thus (\ref{1.6}) is satisfied for every $k\ge 0$.
For $\frac{1}{2} \leq r< \rho \leq R \leq 1 $, we let $\eta(x) $ be a function of class $ C_{0}^{\infty}\left(Q_ {\frac{r+ \rho }{2}}\right)$  with  $\eta=1 $  on  $ Q_r$  and  $|D \eta| \leq \frac{4}{\rho-r } $.  Setting $ \zeta=\eta(u-k)_+$. By the Sobolev Imbedding Theorem, we have
\begin{equation}\label{3.18}
\begin{split}
& \int_ {A_{k,r}}(u-k)^{Q} \mathrm {d} x \le \int_{A_{k,\frac{\rho+r}{2}}} \zeta^{Q} \mathrm {d} x \\
\leq & \Big( \int_ {A_{k,\frac {\rho+r}{2}}} |D(\eta(u-k))|^{Q_* }\mathrm {d} x \Big)^{\frac{Q}{Q_*}} \\
\leq & 2^Q \Big( \int_ {A_{k,\frac{\rho+r}{2}}}| (u-k)D\eta |^{Q_*} \mathrm {d} x + \int_ {A_{k,\frac{\rho+r}{2}}}| \eta Du|^{Q_*} \mathrm {d} x \Big)^{\frac{Q}{Q_*}},
\end{split}
\end{equation}
where $Q_* =\frac{nQ}{n+Q}$. By H\"older inequality with exponents $\frac{Q}{Q_*}$ and  $\frac{Q}{Q-Q_*}$ we obtain
\begin{equation}\label{3.19}
\begin{split}
& \int_ {A_{k,\frac{\rho+r}{2}}}| (u-k)D\eta |^{Q_*} \mathrm {d} x \\
\leq&  c \int_ {A_{k,\rho}}\Big( \frac{u-k}{\rho-r} \Big)^{Q_*} \mathrm {d} x \\
\leq & c \left( \int_ {A_{k,\rho}}\Big(\frac{u-k}{\rho-r}\Big)^{Q} \mathrm {d} x \right)^{\frac{Q_*}{Q}} |A_{k,\rho}|^{1-\frac{Q_*}{Q}}.
\end{split}
\end{equation}
The condition $Q<p^*$ is equivalent to $Q_*<p$. Using H\"older inequality with exponents $\frac p {Q_*}$ and $\frac {p}{p-Q_*}$ and taking into account the inequality (\ref{1.6}) we obtain
\begin{equation}\label{3.20}
\begin{split}
& \int_ {A_{k,\frac{\rho+r}{2}}}|\eta Du|^{Q_*} \mathrm {d} x\\
\leq & c \Big( \int_ {A_{k,\frac{\rho+r}{2}}}|Du|^{p} \mathrm {d} x\Big )^{\frac{Q_*}{p}} |A_{k,\rho}|^{1-\frac{Q_*}{p}} \\
\leq & c \Big(\int_{{A}_{k, \rho}} \Big( \frac{u-k}{\rho-r} \Big )^{Q} \mathrm {d} x +|A_{k, \rho}|^{1-\frac{p}{n}+\varepsilon}\Big)^{\frac{Q_*}{p}}|A_{k,\rho}|^{1-\frac{Q_*}{p}} \\
\leq &  c  \Big( \int_{{A}_{k, \rho}} \Big( \frac{u-k}{\rho-r} \Big )^{Q} \mathrm{d}x \Big) ^{\frac{Q_*}{p}} |A_{k,\rho}|^{1-\frac{Q_*}{p}}+|A_{k,\rho}|^{1- \left(\frac p n -\varepsilon \right) \frac {Q_*}{p} }.
\end{split}
\end{equation}
By (\ref{3.18}), (\ref{3.19}) and (\ref{3.20}) we get for every $ r<\rho \leq R$,
\begin{equation}\label{3.21}
\begin{split}
& \int_ {A_{k,r}}(u-k)^{Q} \mathrm {d} x \\
\leq & c\left\{\Big( \int_ {A_{k,\rho}}\Big(\frac{u-k}{\rho-r}\Big )^{Q} \mathrm {d} x \Big)^{\frac{Q_*}{Q}} |A_{k,\rho}|^{1-\frac{Q_*}{Q}} \right. \\
& \left. +\Big( \int_{{A}_{k, \rho}}\Big( \frac{u-k}{\rho-r} \Big) ^{Q} \mathrm{d}x \Big)^{\frac{Q_*}{p}} |A_{k,\rho}|^{1-\frac{Q_*}{p}}+|A_{k,\rho}|^{1- \left(\frac p n -\varepsilon \right) \frac {Q_*}{p}}
\right\}^{\frac{Q}{Q_*}}  \\
\leq & c \Big( \frac{1}{\rho-r} \Big)^Q \Big(\int_ {A_{k,\rho}}(u-k)^{Q} \mathrm {d} x\Big) |A_{k,\rho}|^{\frac Q n}  \\
& + c \Big( \frac{1}{\rho-r} \Big) ^{\frac{Q^2}{p}} \Big(\int_ {A_{k,\rho}}(u-k)^{Q} \mathrm {d} x \Big)^{\frac{Q}{p}}|A_{k,\rho}|^{1+\frac{Q}{n}-\frac{Q}{p}}+c |A_{k,\rho}|^{1+\frac{Q}{p} \varepsilon}.
\end{split}
\end{equation}
Denote
$$
U(k,t) =\int_{A_{k,t}} (u-k) ^Q {\rm d} x.
$$
It is obvious that $U(\cdot, \rho)$ is non-increasing and $U(k,\cdot)$ is non-decreasing. For each $h<k$ and $r < \rho$, one has
\begin{equation}\label{1.12222}
\begin{array}{llll}
\displaystyle
U(h,\rho)
&\displaystyle=\int_{A_{h,\rho}} (u-h)^Q {\rm d}x \ge \int_{A_{k,\rho}} (u-h)^Q {\rm d}x\\
& \displaystyle\ge (k-h) ^Q |A_{k,\rho}| \ge (k-h) ^Q |A_{k,r}| .
\end{array}
\end{equation}
It is no loss of generality to assume that
$$
U(k,t) \le 1 \ \mbox { and } \ |A_{k,t}|\le 1.
$$
We take $\varepsilon$ small enough such that
\begin{equation}\label{1.14}
\varepsilon< \frac p Q \left( 1-\frac{Q}{p*}\right),
\end{equation}
this implies
\begin{equation}\label{1.15}
{1+\frac{Q}{n}-\frac{Q}{p}}>{\frac{Q}{p} \varepsilon}.
\end{equation}
Let $\tau $ satisfy
\begin{equation*}\label{definitioin for tau}
\frac{np-Q}{p(n-1)}n\varepsilon=Q\tau.
\end{equation*}
Using (\ref{1.14}) and (\ref{1.15}), we get from (\ref{3.21}) and (\ref{1.12222})  that
\begin{equation}\label{3.1111}
\begin{split}
U(k, r) & \leq c\Big( \frac{1}{\rho-r} \Big )^Q U(k, \rho) |A_{k,\rho}|^{\frac{Q}{n}} \\
& + c \Big( \frac{1}{\rho-r} \Big) ^{\frac{Q^2}{p}} U(k, \rho)^{\frac{Q}{p}}|A_{k,\rho}|^{1+\frac{Q}{n}-\frac{Q}{p}}+c |A_{k,\rho}|^{1+\frac{Q}{p} \varepsilon} \\
& \leq c\Big( \frac{1}{\rho-r} \Big )^Q U(h, \rho) |A_{h,\rho}|^{\frac{Q}{n}} \\
& \quad   +c \Big( \frac{1}{\rho-r} \Big) ^{\frac{Q^2}{p}} U(h, \rho)^{\frac{Q}{p}}|A_{h,\rho}|^{1+\frac{Q}{n}-\frac{Q}{p}}+c|A_{h,\rho}|^{\frac{Q}{p} \varepsilon}\cdot|A_{k,\rho}| \\
& \leq c \left( \frac{1}{\rho-r} \right )^{\frac{Q^2}{p}} U(h, \rho) |A_{h,\rho}|^{1+\frac{Q}{n}-\frac{Q}{p}}+c |A_{h,\rho}|^{\frac{Q}{p} \varepsilon}U(h, \rho) \Big(\frac{1}{k-h}\Big )^Q\\
& \leq c \left[\Big( \frac{1}{\rho-r} \Big)^{\frac{Q^2}{p}} +\Big( \frac{1}{k-h} \Big)^Q \right]U(h, \rho) |A_{h,\rho}|^{\frac{Q}{p} \varepsilon}  \\
& \leq c \left[ \Big(\frac{\rho^{\frac{p\tau}{Q}}}{\rho-r}  \Big)^{\frac{Q^2}{p}} +  \Big( \frac{\rho^{\tau}}{k-h}  \Big)^Q \right] \rho^{-\frac{np-Q}{p(n-1)}n\varepsilon} U(h, \rho) |A_{h,\rho}|^{\varepsilon}.
\end{split}
\end{equation}
Since
$$
\rho^{\frac{p\tau}{Q}}<\rho+1,
$$
then
$$
\Big(\frac{\rho^{\frac{p\tau}{Q}}}{\rho-r}  \Big)^{\frac{Q^2}{p}} \leq  c  \left\{\Big(\frac{\rho}{\rho-r} \Big)^{\frac{Q^2}{p}}+  \Big(\frac{1}{\rho-r} \Big)^{\frac{Q^2}{p}}\right\},
$$
thus (\ref{3.1111}) implies
\begin{equation}\label{3.1112}
 U(k, r)\leq c \left[ \Big(\frac{\rho}{\rho-r} \Big)^{\frac{Q^2}{p}} + \Big(\frac{1}{\rho-r} \Big)^{\frac{Q^2}{p}} + \Big( \frac{\rho^{\tau}}{k-h}  \Big)^Q \right] \rho^{-\frac{np-Q}{p(n-1)}n\varepsilon} U(h, \rho) |A_{h,\rho}|^{\varepsilon}.
\end{equation}
From (\ref{1.12222}) we know that
$$
|A_{k,r}| ^\gamma \le \frac { U(h,\rho) ^\gamma }{ (k-h) ^{Q\gamma}},
$$
which together with (\ref{3.1112}) yields
\begin{equation*}
\begin{split}
& U(k, r) |A_{k, r}|^{\gamma} \\
\leq & c \left[ \left( \frac{\rho}{\rho-r} \right)^{\frac{Q^2}{p}} + \Big(\frac{1}{\rho-r} \Big)^{\frac{Q^2}{p}} + \left( \frac{\rho^{\tau}}{k-h} \right)^Q \right] \rho^{-\frac{np-Q}{p(n-1)}n\varepsilon} U(h, \rho)^{1+\gamma}{\frac {1}{(k-h)^{Q\gamma}}} |A_{h,\rho}|^{\varepsilon}.
\end{split}
\end{equation*}
Let us now choose $\gamma$  in such a way that $ \gamma(1+\gamma)=\varepsilon $, and let us define
$$
\varphi(k, t)=U(k, t)|A_{k, t}|^{\gamma}.
$$
For $r<\rho\leq R$  and  $h<k$ we have
\begin{equation}\label{3.23}
\begin{split}
\varphi(k, r) \leq c \left[ \left( \frac{\rho}{\rho-r} \right)^{\frac{Q^2}{p}} + \Big(\frac{1}{\rho-r} \Big)^{\frac{Q^2}{p}} +\left( \frac{\rho^{\tau}}{k-h} \right)^Q \right] \rho^{-\frac{np-Q}{p(n-1)}n\varepsilon} {\frac {\varphi(h, \rho)^{1+\gamma}}{(k-h)^{Q\gamma}}}.
\end{split}
\end{equation}
Let now $ d \geq C R^{\tau}$ be a constant that we shall fix later, and define
$$
k_{i} =d\left(1-2^{-i}\right)
$$
and
$$
r_{i} =\frac{R}{2}\left(1+2^{-i}\right).
$$
In (\ref{3.23}) we choose
$$
r=r_{i+1},\ \rho=r_{i},\  k=k_{i+1},\ h=k_{i},
$$
then
$$
\rho-r=r_{i}-r_{i+1}=\frac{R}{2}\frac{1}{2^{i+1}},
$$
$$
k-h=k_{i+1}-k_{i}=d\frac{1}{2^{i+1}}.
$$
Let us define $\varphi _i =\varphi (k_i,r_i)$, then
\begin{equation*}
\begin{split}
\varphi_{i+1}  & \leq  c \left\{ \left[ \frac{\frac{R}{2}(1+2^{-i})}{{\frac{R}{2}}{\frac{1}{2^{i+1}}}}\right]^{\frac {Q^2}{p}} + \left[ \frac{1}{{\frac{R}{2}}{\frac{1}{2^{i+1}}}}\right]^{\frac {Q^2}{p}}+ \left[ \frac{(\frac{R}{2}(1+2^{-i}))^{\tau}}{d{\frac{1}{2^{i+1}}}} \right]^{Q} \right\} \times \\
& \qquad \qquad  \times \left[\frac{R}{2}(1+2^{-i})\right]^{-\frac{np-Q}{p(n-1)}n\varepsilon} \frac {\varphi_i^{1+\gamma}}{({\frac{d}{2^{i+1}}})^{Q\gamma}}\\
&\leq c d^{-Q \gamma} 2^{Q i(\frac{Q}{p}+\gamma)} R^{-\frac{np-Q}{p(n-1)} n\varepsilon} \varphi_{i}^{1+\gamma}.
\end{split}
\end{equation*}
We can now apply Lemma \ref{lemma 1.81} with
$$
C=c d^{-Q \gamma} R^{-\frac{np-Q}{p(n-1)} n\varepsilon}>0,\ B=2^{Q(\frac{Q}{p}+\gamma)}>1,\ \alpha=\gamma.
$$
Choosing
$$
d \geq c R^{-\frac{np-Q}{p(n-1)}n\varepsilon \frac{1}{Q\gamma}} \varphi_{0}^{\frac{1}{Q}}
$$
with the constant $c$  large enough, we can conclude that the sequence  $\varphi_{i}$  tends to zero, and hence
$$
\varphi\left(d, \frac{R}{2}\right)=0.
$$
The condition imposed on $d$  will be satisfied by taking
$$d=c R^{\tau}+c R^{-\frac{np-Q}{p(n-1)}n\varepsilon \frac{1}{Q\gamma}} \varphi_{0}^{\frac{1}{Q}},$$
and hence, recalling the choice of  $ \gamma$, we arrive at
$$
\sup _{B_{\frac{R}{2}}} u \leq d = c \left(\frac{1}{R^{\frac{(np-Q)n }{p(n-1)}}} \int_{A_{{0}, R}}u^{Q} d x\right)^{\frac{1}{Q}}\left(\frac{\left|A_{0, R}\right|}{R^{\frac{(np-Q)n }{p(n-1)}}}\right)^{\frac{\gamma}{Q}} +c R^{\tau}.
$$
The conclusion follows at once writing  $u-\kappa_{0}$  instead of  $u$.
\end{proof}

\vspace{3mm}

\begin{lemma}\label{lemma 4}
Let $u$ be a locally bounded function, satisfying (\ref{1.6}) (with  $p>1$) for every  $k \ge \kappa_0\ge 0$, and let
$$
2 k_{0}=M(2 R)+m(2 R)=: \sup _{Q_{2 R}} u+\inf _{Q_{2 R}} u.
$$
Assume that $ \left|A_{k_0,R}\right| \leq \lambda\left|Q_{R}\right|$  for some $ \lambda<1$. If for an integer $ \nu $, it holds that
\begin{equation}\label{3.24-1}
\operatorname{osc}(u, 2 R) \geq c 2^{\nu+1}  R^{\tau},
\end{equation}
where $\operatorname{osc}(u, 2 R)$ is the oscillation of the function $u$ over $Q_{2R}$, then, setting
\begin{equation}\label{3.24-2}
k_{\nu}=M(2 R)-2^{-\nu-1} \operatorname{osc}(u, 2 R) ,
\end{equation}
we have
\begin{equation}\label{3.25}
\left|A_{k_{\nu}, R}\right| \leq c \nu^{-\frac{n(p-1)}{p(n-1)}} R^{\frac{(np-Q)n}{p(n-1)}}.
\end{equation}
\end{lemma}

\begin{proof}
 For $ k_{0}<h<k $  let us define
\begin{equation*}
v(x)=\left\{\begin{array}{lll}
k-h & \text { if } & u \geq k, \\
u-h & \text { if } & h<u<k, \\
0 & \text { if } & u \leq h .
\end{array}\right.
\end{equation*}
We have $ v=0$  in $ B_{R}\setminus A_{k_0, R} $, and since the measure of this set is greater than $ (1-\lambda)\left|B_{R}\right| $, we can apply the Sobolev inequality, obtaining
$$ \left(\int_{Q_{R}} v^{\frac{n}{n-1}} d x\right)^{1-\frac{1}{n}} \leq c \int_{\Delta}|D v| d x=c \int_{\Delta}|D u| d x$$
in which
$$
\Delta=A_{h,R}\setminus A_{k,R}.
$$
We therefore have
\begin{equation}\label{3.26}
(k-h)|A_{k, R}|^{1-\frac{1}{n}}\le \left(\int_{Q_{R}} v^{\frac{n}{n-1}} d x\right)^{1-\frac{1}{n}} \leq c|\Delta|^{1-\frac{1}{p}}\left(\int_{A_{h, R}}|D u|^{p} d x\right)^{\frac{1}{p}}.
\end{equation}
On the other hand, from (\ref{1.6}) we deduce
\begin{equation*}\label{3.27}
\begin{aligned}
 \displaystyle \int_{A_{h, R}}|D u|^{p} d x \leq &\displaystyle  \frac{c}{R^{Q}} \int_{A_{h,2R}}(u-h)^{Q} d x+c |{A_{h, 2R}}|^{1-\frac{p}{n}+\varepsilon} \\
\leq &\displaystyle  c R^{n-Q}(M(2 R)-h)^{Q}+c  R^{n-p+n \varepsilon}.
\end{aligned}
\end{equation*}
For $ h \leq k_{\nu}$, (\ref{3.24-1}) and (\ref{3.24-2}) merge into
$$
M(2 R)-h \geq M(2 R)-k_{\nu}  \geq 2^{-\nu-1} \operatorname{osc}(u, 2 R) \geq  c R^{\tau},
$$
therefore
$$
(M(2 R)-h)^Q \geq c R^{Q \tau}=cR^{n\varepsilon \frac{(np-Q)}{p(n-1)}} \geq cR^{n\varepsilon},
$$
$$
R^{n-Q}\geq R^{n-p},
$$
$$
R^{n-Q}(M(2 R)-h)^{Q}\geq c  R^{n-p+n \varepsilon},
$$
and hence
\begin{equation}\label{3.26-2}
\displaystyle \int_{A_{h, R}}|D u|^{p} d x \leq 2c R^{n-Q}(M(2 R)-h)^{Q}.
\end{equation}
Substituting (\ref{3.26-2}) into (\ref{3.26}), one has
$$
(k-h)|A_{k,R}|^{1-\frac{1}{n}} \leq c|\Delta|^{1-\frac{1}{p}} R^{\frac{n-Q}{p}}(M(2 R)-h)^{\frac{Q}{p}} .
$$
Writing the above inequality for the levels
$$
k=k_{i}=M(2 R)- 2^{-i-1} \operatorname{osc}(u, 2 R)
$$
and
$$
h=k_{i-1}=M(2 R)- 2^{-i} \operatorname{osc}(u, 2 R),
$$
we get
\begin{equation*}
\begin{array}{llll}
&\displaystyle \frac 1 2 \frac{1}{2^{i}}\operatorname{osc}(u, 2 R)|A_{k_{i-1}, R}|^{1-\frac{1}{n}}\\
\leq &\displaystyle \frac{1}{2^{i+1}}\operatorname{osc}(u, 2 R)|A_{k_i, R}|^{1-\frac{1}{n}} \\
\leq &\displaystyle c |\Delta_i|^{1-\frac{1}{p}} R^{\frac{n-Q}{p}} \left(\frac{1}{2^{i}}\operatorname{osc}(u, 2 R)\right)^{\frac {Q}{p}},
\end{array}
\end{equation*}
the inequality above implies
\begin{equation*}
|A_{k_\nu, R}|^{1-\frac{1}{n}}\leq |A_{k_i, R}|^{1-\frac{1}{n}}\leq c |\Delta_i|^{1-\frac{1}{p}} R^{\frac{n-Q}{p}} \left(\frac{1}{2^{i}}\operatorname{osc}(u, 2 R)\right)^{\frac {Q}{p}-1},
\end{equation*}
where $ \Delta_{i}=A _{k_{i-1}, R} \setminus A_{k_{i}, R} $. Raising both sides of the above inequality to the power $ \frac{p}{p-1} $ one gets
$$
\left|A_ {k_{\nu}, R}  \right|^{\frac{p(n-1)}{n(p-1)}}  \leq c R^{\frac{n-Q}{p-1}}\left|\Delta_{i}\right|\left(\frac{1}{2^{i}}\operatorname{osc}(u, 2 R)\right)^{\frac {Q-p}{p-1}} .
$$
We now sum over  $i$  from 1 to $\nu$,
\begin{equation*}
\begin{split}
& \sum _{i=1} ^\nu  \left|\Delta_{i}\right|\left(\frac{1}{2^{i}}\operatorname{osc}(u, 2 R)\right)^{\frac {Q-p}{p-1}}\\
\leq & |A_{k_{0},R}|\left(\frac{1}{2}\operatorname{osc}(u, 2 R)\right)^{\frac {Q-p}{p-1}} - |A_{k_{1},R}|\left(\frac{1}{2}\operatorname{osc}(u, 2 R)\right)^{\frac {Q-p}{p-1}} \\
& +|A_{k_{1},R}|\left(\frac{1}{2^2}\operatorname{osc}(u, 2 R)\right)^{\frac {Q-p}{p-1}} -|A_{k_{2},R}|\left(\frac{1}{2^2}\operatorname{osc}(u, 2 R)\right)^{\frac {Q-p}{p-1}}
\end{split}
\end{equation*}
\begin{equation*}
\begin{split}
& + \cdots \\
& + |A_{k_{\nu-1},R}|\left(\frac{1}{2^\nu}\operatorname{osc}(u, 2 R)\right)^{\frac {Q-p}{p-1}} - |A_{k_{\nu},R}|\left(\frac{1}{2^\nu}\operatorname{osc}(u, 2 R)\right)^{\frac {Q-p}{p-1}}\\
\leq & |A_{k_{0},R}|\left(\frac{1}{2}\operatorname{osc}(u, 2 R)\right)^{\frac {Q-p}{p-1}} - |A_{k_{1},R}|\left(\frac{1}{2^2}\operatorname{osc}(u, 2 R)\right)^{\frac {Q-p}{p-1}} \\
& +|A_{k_{1},R}|\left(\frac{1}{2^2}\operatorname{osc}(u, 2 R)\right)^{\frac {Q-p}{p-1}} -|A_{k_{2},R}|\left(\frac{1}{2^3}\operatorname{osc}(u, 2 R)\right)^{\frac {Q-p}{p-1}}\\
& + \cdots \\
& + |A_{k_{\nu-1},R}|\left(\frac{1}{2^\nu}\operatorname{osc}(u, 2 R)\right)^{\frac {Q-p}{p-1}} - |A_{k_{\nu},R}|\left(\frac{1}{2^\nu}\operatorname{osc}(u, 2 R)\right)^{\frac {Q-p}{p-1}}\\
\leq & |A_{k_{0},R}|\left(\frac{1}{2}\operatorname{osc}(u, 2 R)\right)^{\frac {Q-p}{p-1}}- |A_{k_{\nu},R}|\left(\frac{1}{2^\nu}\operatorname{osc}(u, 2 R)\right)^{\frac {Q-p}{p-1}}\\
\leq & \left|A_{k_{0},R}\right| \left(\frac{1}{2}\operatorname{osc}(u, 2 R)\right)^{\frac {Q-p}{p-1}}\\
\leq &  c \left|A_{k_{0},R}\right|,
\end{split}
\end{equation*}
where we have used the fact
$$
 \operatorname{osc}(u, 2 R) \le \mbox {esssup}_{B_{2R}} u <\infty.
$$
Thus
\begin{equation*}
\nu\left|A_{k_{\nu}, R}\right|^{\frac{p(n-1)}{n(p-1)}} \leq c R^{\frac{n-Q}{p-1}}\left|A_{k_{0}, R}\right| \leq c R^{\frac{pn-Q}{(p-1)}},
\end{equation*}
and so
\begin{equation*}
\left|A_{ k_{\nu}, R}\right| \leq c \nu^{-\frac{n(p-1)}{p(n-1)}} R^{\frac{(np-Q)n}{p(n-1)}},
\end{equation*}
as desired.
\end{proof}

\vspace{3mm}
In the proof of Theorem \ref{theorem 21} we shall need the following algebraic lemma, which comes from Lemma 7.3 in \cite{Giusti}.
\begin{lemma}\label{the new lemma}
Let $ \varphi(t) $ be a positive function, and assume that there exists a constant  $q$  and a number  $ \tilde{\tau}$, $0<\tilde{\tau}<1 $ such that for every $ R<R_{0}$
$$\varphi(\tilde{\tau} R) \leq \tilde{\tau}^{\delta} \varphi(R)+B R^{\beta}$$
with $ 0<\beta<\delta ,$ and
$$
\varphi(t) \leq q \varphi\left(\tilde{\tau}^{k} R\right)
$$
for every $ t $ in the interval $ \left(\tilde{\tau}^{k+1} R, \tilde{\tau}^{k} R\right) $.
Then, for every $\varrho<R \leq R_{0} $ we have
$$
\varphi(\varrho) \leq C\left\{\left(\frac{\varrho}{R}\right)^{\beta} \varphi(R)+B \varrho^{\beta}\right\},
$$
where $ C$  is a constant depending only on  $q, \tilde{\tau}, \delta$ and $ \beta$.
\end{lemma}

With the above Lemmas in hands, we can now prove Theorem \ref{theorem 21}.

\vspace{3mm}

\begin{proof}  Let, as above, $2 k_{0}=M(2 R)+m(2 R) $. We can assume without loss of generality that
$$
\left|A_{k_{0}, R} \right| \leq \frac{1}{2}\left|Q_{R}\right|,$$
since otherwise we would have
$$
\left|B_{k_{0}, R}\right|=\left|Q_{R}\right|-\left|A_{ k_{0}, R} \right| \leq \frac{1}{2}\left|Q_{R}\right|,
$$
and it will be sufficient to write  $-u $ instead of $ u $.

Setting $ k_{\nu}=M(2 R)-2^{-\nu-1} \operatorname{osc}(u, 2 R) $, we have $ k_{\nu}> \kappa _{0} $. Write (\ref{3.17})  with $ k_{\nu} $ instead of $ \kappa _{0} $
\begin{equation*}\label{3.28}
\begin{split}
\sup _{B_{\frac{R}{2}}} u-k_{\nu}& \leq c \left(\frac{1}{R^{\frac{(np-Q)n}{p(n-1)}}} \int_{A_{k_{\nu}, R}}\left(u-k_{\nu}\right)^{Q} d x\right)^{\frac{1}{Q}}\left(\frac{\left|A_{k_{\nu}, R}\right|}{R^{\frac{(np-Q)n}{p(n-1)}}}\right)^{\frac{\gamma}{Q}} +c R^{\tau}\\
& \leq c \sup _{B_{R}}\left(u-k_{\nu}\right)\left(\frac{\left|A_{k_{\nu}, R}\right|}{R^{ \frac{(np-Q)n}{p(n-1)}}}\right)^{\frac{\gamma+1}{Q}}+c  R^{\tau}.
\end{split}
\end{equation*}
Let us now choose the integer  $\nu$  in such a way that
$$
c \nu^{-\frac{n(p-1)}{p(n-1)}} \leq \frac{1}{2}.
$$
If $ \operatorname{osc}(u, 2 R) \geq c 2^{\nu+1}  R^{\tau} $, we deduce from (\ref{3.25}) that
$$
\sup _{B_{\frac{R}{2}}} u-k_{\nu} \leq \frac{1}{2} \sup _{B_{R}}\left(u-k_{\nu}\right) +c  R^{\tau},
$$
therefore
$$
 M\left(\frac{R}{2}\right)-k_{\nu} \leq \frac{1}{2}\left(M(2 R)-k_{\nu}\right)+c R^{\tau},
 $$
so that, subtracting from both members the quantity  $ m\left(\frac{R}{2}\right)$,
$$
M\left(\frac{R}{2}\right)-m\left(\frac{R}{2}\right)-k_{\nu} \leq \frac{1}{2}\left(M(2 R)-k_{\nu}\right)-m\left(\frac{R}{2}\right)+c R^{\tau},
$$
then
$$
M\left(\frac{R}{2}\right)-m\left(\frac{R}{2}\right) \leq \frac{1}{2}M(2 R)-m\left(\frac{R}{2}\right)+\frac{1}{2}k_{\nu}+c R^{\tau},
$$
by (\ref{3.24-2}) we have
\begin{equation*}
\begin{split}
\operatorname{osc}\left(u, \frac{R}{2}\right) & \leq \frac{1}{2}M(2 R)-m\left(\frac{R}{2}\right)+\frac{1}{2}\left[M(2R) -\frac{1}{2^{\nu+1}}\operatorname{osc}\left(u, 2R\right) \right]+c R^{\tau} \\
& \leq M(2 R)-m(2 R)-\frac{1}{2^{\nu+2}}\operatorname{osc}\left(u, 2R \right)+c R^{\tau}\\
& \leq\left(1-\frac{1}{2^{\nu+2}}\right) \operatorname{osc}(u, 2 R)+cR^{\tau}.
\end{split}
\end{equation*}
In conclusion, either the function $ \operatorname{osc}(u, R) $ satisfies the above relation, or else
$$\operatorname{osc}(u, 2 R) \leq  c 2^{\nu+1} R^{\tau}.
$$
In any case, we have
\begin{equation*}\label{3.29}
\operatorname{osc}\left(u, \frac{R}{2}\right) \leq\left(1-\frac{1}{2^{\nu+2}}\right) \operatorname{osc}(u, 2 R)+c 2^{\nu} R^{\tau}.
\end{equation*}
We can now apply the preceding Lemma \ref{the new lemma} with $ \tilde{\tau}=1 / 4 $ and $ \delta=\log _{\tilde{\tau}}(1-(2^{-\nu-2})) $. Decreasing if necessary the value of $ \beta $, we can assume that $ \beta<\delta $.
We therefore have
$$
\operatorname{osc}(u, \varrho) \leq c\left\{\left(\frac{\varrho}{R}\right)^{\beta} \operatorname{osc}(u, R)+ c \varrho^{\beta}\right\},
$$
for every $ \varrho<R \leq \min \left(R_{0}, \operatorname{dist}\left(x_{0}, \partial \Omega\right)\right)$. The above inequality shows that $u$ is locally H\"older continuous.
This ends the proof of Theorem \ref{theorem 21}.
\end{proof}

\vspace{2mm}

We note in Definition \ref{definition 2} that, a function $u\in W_{loc} ^{1,p} (\Omega)$ belongs to the generalized De Giorgi class $GDG^+_p (\Omega, p ,Q,y, y_*,\varepsilon, \kappa _0)$ if it satisfies (\ref{eDG+}). For a vector valued function $u=(u^1,\cdots, u^N)\in W_{loc} ^{1,p} (\Omega, \mathbb R^N)$, $N\ge 1$, one can give a similar definition as follows.

\begin{definition}\label{definition 2.2}
We say that $u\in W_{loc} ^{1,p} (\Omega, \mathbb R^N)$, $N\ge 1$, belongs to the generalized De Giorgi class $GDG_{p}^+(N, \Omega, p, Q, y, y_*,\varepsilon ,\kappa_0)$, $1< p\le n $, $p\le Q<p^*$, $y$ and $\varepsilon>0$, $y_*$ and $\kappa _0\ge 0$,  if
\begin{equation}\label{eDG+2}
\sum_{\alpha=1}^{N}\int_{A_{k,\sigma \rho }^{\alpha}}\left|D u^{\alpha}\right|^{2}  {\rm d}x  \leq y \sum_{\alpha=1}^{N} \int_{A_{k,\rho }^{\alpha}}\left(\frac{u^{\alpha}-k}{(1-\sigma )\rho }\right)^Q  {\rm d}x  + y_* \sum_{\alpha=1}^{N}|A_{k,\rho }^{\alpha}|^{1-\frac{p}{n}+\varepsilon},
\end{equation}
for all $k\ge \kappa_0 $, $\sigma \in (0,1)$, and all pairs of concentric cubes $Q_{\sigma \rho} (x_0) \subset Q_\rho (x_0) \subset \Omega$ centered at $x_0$, where
$$
A_{k,\rho }^{\alpha} =\{x\in \Omega: u^\alpha >k\}\cap Q_\rho.
$$
\end{definition}

We can define similarly $GDG_{p}^- (N, \Omega, p, Q,y, y_*,\varepsilon ,\kappa_0)$ to be the class of functions $u$ such that $-u \in GDG_p^+ (N, \Omega, p, Q,y, y_*,\varepsilon ,\kappa_0)$. More explicitly, they are the vectors in $W_{loc} ^{1,p} (\Omega, \mathbb R^N)$ such that for all $k\le -\kappa _0$, all $\sigma \in (0,1)$, and all pairs of concentric cubes $Q_{\sigma \rho}\subset Q_\rho \subset \Omega$,
\begin{equation}\label{eDG+22}
\sum_{\alpha=1}^{N}\int_{B_{k,\sigma \rho }^{\alpha}}\left|D u^{\alpha}\right|^{2}  {\rm d}x  \leq y \sum_{\alpha=1}^{N} \int_{B _{k,\rho }^{\alpha}}\left(\frac{k-u^{\alpha}}{(1-\sigma )\rho }\right)^Q  {\rm d}x  + y_* \sum_{\alpha=1}^{N}|B_{k,\rho }^{\alpha}|^{1-\frac{p}{n}+\varepsilon},
\end{equation}
where
$$
B_{k,\rho }^{\alpha} =\{x\in \Omega: u^\alpha<k\}\cap Q_\rho.
$$
We shall indicate by $GDG_{p}(N,\Omega, p, Q,y, y_*,\varepsilon ,\kappa_0)$ the class of functions belonging both to $GDG_{p}^+$ $(N,\Omega, p, Q,y, y_*,\varepsilon ,\kappa_0)$ and $GDG_{p}^-(N,\Omega, p, Q,y, y_*,\varepsilon ,\kappa_0)$:
\begin{equation}\label{definition for GDG}
\begin{array}{llll}
&\displaystyle GDG_{p} (N,\Omega, p, Q,y, y_*,\varepsilon ,\kappa_0)\\
= &\displaystyle GDG_{p}^+(N,\Omega, p, Q,y, y_*,\varepsilon ,\kappa_0) \cap GDG_{p}^-(N,\Omega, p, Q,y, y_*,\varepsilon ,\kappa_0).
\end{array}
\end{equation}
It is clear that
$$
GDG_{p} (1,\Omega, p, Q,y, y_*,\varepsilon ,\kappa_0) =GDG_{p} (\Omega, p, Q,y, y_*,\varepsilon ,\kappa_0)
$$
and
$$
GDG_{p}(1, \Omega, p, p,y, y_*,\varepsilon ,\kappa_0) =DG_p(\Omega, p,y, y_*,\varepsilon ,\kappa_0).
$$

One can prove, similar to the proof of Theorem \ref{theorem 21}, that
\begin{theorem}\label{theorem 212}
Let $u \in GDG_{p} (N,\Omega, p, Q,y, y_*,\varepsilon,\kappa_0)$ for $1<p\le n$, $N\ge 1$ and some $Q \in [p, p^*)$, then $u$ is locally bounded and locally H\"older continuous in $\Omega$.
\end{theorem}

For the proof of Theorem \ref{theorem 212}, One can repeat the proof of Theorem \ref{theorem 21} without any difficulty, we omit the details.

In the next four sections we shall give applications of Theorems \ref{theorem 21} and \ref{theorem 212}. In Section 3 we shall consider a polyconvex integral functional in four-dimensional Euclidean spaces with the integrand has splitting form. Under some structural conditions on the energy density, we prove that all local minimizers are locally bounded and locally H\"older continuous. In Section 4 we shall consider a special type of linear elliptic equation with degenerate coercivity, under suitable integrability assumption on the coefficient, we derive that any of its weak solutions is locally bounded and locally H\"older continuous. In Section 5, we shall show that our Theorem \ref{theorem 21} can be applied in dealing with regularity issues of elliptic equations with non-standard grow conditions. In the last section, Section 6, we treat with quasilinear elliptic systems. Under suitable assumptions on the coefficients, we show that any of its weak solutions is locally bounded and locally H\"older continuous.


\section{A polyconvex integral functional.}

In this section we give an application of Theorem \ref{theorem 21} to regularity property for minimizers of some polyconvex integrals in four dimensional Euclidean spaces with the
integrand has splitting structure. More precisely, let $\Omega$ be an open bounded subset of $\mathbb R^4$ and let us consider the variational integral
\begin{equation}\label{integral functionals}
\mathcal{F}(u,\Omega)=\int_{\Omega}f(x, Du (x))dx,
\end{equation}
where $f:\Omega\times\mathbb R^{4\times 4}\rightarrow \mathbb R$ is a Carath\'eodory function (that is, measurable with respect to $x$ for every $\xi\in \mathbb R^{4\times 4}$ and continuous with respect to $\xi$ for almost every $x\in \Omega$),
$$
u =(u^1,u^2,u^3,u^4)^t:\Omega\subset\mathbb R^4 \rightarrow\mathbb R^4
$$
is a vector-valued map, and $Du$ is the $4\times 4$ Jacobian matrix of its partial derivatives, i.e.,
$$
Du=\left(
\begin{array}{c}
Du^1\\
Du^2\\
Du^3\\
Du^4\\
\end{array}
\right)=\left(
\begin{array}{cccc}
D_{1}u^{1} &D_{2}u^{1}   &D_{3}u^{1}   &D_{4}u^{1}\\
D_{1}u^{2} &D_{2}u^{2}   &D_{3}u^{2}   &D_{4}u^{2}\\
D_{1}u^{3} &D_{2}u^{3}   &D_{3}u^{3}   &D_{4}u^{3}\\
D_{1}u^{4} &D_{2}u^{4}   &D_{3}u^{4}   &D_{4}u^{4}
\end{array}
\right), \; D_{\beta}u^{\alpha}=\frac{\partial u^{\alpha}}{\partial x_{\beta}}, \; \alpha,\beta\in\{1, 2, 3,4\}.
$$
We assume that $f:\Omega \times \mathbb R^{4\times 4}\rightarrow \mathbb R$ has splitting form:
\begin{equation}\label{2.2}
f(x,\xi)= \sum _{\alpha =1} ^4 F_\alpha \big( x,\xi ^\alpha \big)+ \sum_{\beta=1}^6 G_\beta \big( x, (\mbox {adj}_{2}\xi)^\beta \big)+ \sum_{\gamma=1}^4 H_\gamma \big( x,(\mbox {adj}_{3}\xi)^\gamma \big) +I(x, \det \xi),
\end{equation}
where
$$
F_{\alpha} (x,\lambda ): \Omega\times\mathbb R^{4}\rightarrow \mathbb R, \ \ \alpha =1,2,3,4,
$$
$$
G_{\beta} (x,\eta):\Omega\times\mathbb R^{6}\rightarrow \mathbb R, \ \ \beta =1,2,3,4,5,6,
$$
$$
H_{\gamma} (x,\lambda):\Omega\times\mathbb R^{4}\rightarrow \mathbb R,\ \ \gamma =1,2,3,4,
$$
and
$$
I (x,t):\Omega\times\mathbb R\rightarrow \mathbb R,
$$
are Carath\'eodory functions such that $\lambda \mapsto F_\alpha(x,\lambda)$, $\eta \mapsto G_\beta(x,\eta)$, $\lambda \mapsto H_\gamma (x,\lambda)$ and $t \mapsto I(x,t)$ are convex.

In (\ref{2.2}),
\begin{equation*}
\xi=
\left(
\begin{array}{cccc}
\xi^1\\
\xi^2\\
\xi^3\\
\xi^4
\end{array}
\right)=
\left(
\begin{array}{cccc}
\xi_{1}^{1} &\xi_{2}^{1}   &\xi_{3}^{1}   &\xi_{4}^{1}\\
\xi_{1}^{2} &\xi_{2}^{2}   &\xi_{3}^{2}   &\xi_{4}^{2}\\
\xi_{1}^{3} &\xi_{2}^{3}   &\xi_{3}^{3}   &\xi_{4}^{3}\\
\xi_{1}^{4} &\xi_{2}^{4}   &\xi_{3}^{4}   &\xi_{4}^{4}
\end{array}
\right)\in \mathbb R^{4\times 4},
\end{equation*}
$\xi ^\alpha =(\xi ^\alpha_1,\xi ^\alpha_2,\xi ^\alpha_3,\xi ^\alpha_4 )$ is the $\alpha$th row of $\xi$, $\alpha =1,2,3,4$; $\mbox {adj}_2 \xi \in \mathbb R^{6\times 6}$ denote the adjugate matrix of order 2 whose components are
\begin{equation}\label{components1}
({\rm adj}_2 \xi) ^j_k= (-1)^{j_{1}+j_{2}+k_{1}+k_{2}} \det \left(
\begin{array}{cc}
\xi^{j_1}_{k _1} &\xi ^{j_1}_{k _2}\\
\xi^{j_2}_{k_1} &\xi ^{j_2}_{k_2}
\end{array}
\right), \ j,k=1,2,3,4,5,6,
\end{equation}
where we have denoted
\begin{equation*}
\left\{
\begin{array}{lll}
1_1=1\\
1_2=2
\end{array} ,
\right.
\left\{
\begin{array}{lll}
2_1=1\\
2_2=3
\end{array},
\right.
\left\{
\begin{array}{lll}
3_1=1\\
3_2=4
\end{array},
\right.\left\{
\begin{array}{lll}
4_1=2\\
4_2=3
\end{array},
\right.\left\{
\begin{array}{lll}
5_1=2\\
5_2=4
\end{array},
\right.
\left\{
\begin{array}{lll}
6_1=3\\
6_2=4
\end{array},
\right.
\end{equation*}
and
$$
(\mbox {adj}_2 \xi)^\beta =\big( (\mbox {adj}_2 \xi)^\beta_1,\cdots, (\mbox {adj}_2 \xi)^\beta_6 \big)\in \mathbb R^6, \ \ \beta =1,2,3,4,5, 6,
$$
is the $\beta$th row of $\mbox {adj}_2 \xi \in \mathbb R^{6\times 6}$. We note that, for $j=1,2,3$, since $j_1=1$, then $({\rm adj}_2 \xi) ^j$ depends on the entries of the first row of $\xi$, while for $j=4,5,6$, since $j_1\ne 1$ and $j_1<j_2$, then $({\rm adj}_2 \xi) ^j$ does not depend on the entries of the first row of $\xi$, we shall use these facts in the sequel;
$\mbox {adj}_{3}\xi \in \mathbb R^{4\times 4} $ denotes the adjugate matrix of order 3 whose components are
\begin{equation}\label{components 2}
(\mbox {adj}_{3}\xi )^{\gamma }_j =(-1)^{\gamma+j} \det \left(
\begin{array}{ccc}
\xi^\varepsilon_k &\xi ^\varepsilon _\ell &\xi ^\varepsilon _m \\
\xi^\delta  _k &\xi ^\delta  _\ell &\xi ^\delta  _m \\
\xi^\tau  _k &\xi ^\tau  _\ell &\xi ^\tau  _m
\end{array}
\right),\;\;\;  \gamma, j\in\{1,2,3,4\},
\end{equation}
where $\varepsilon, \delta ,\tau  \in\{1,2,3,4\}\backslash\{\gamma\}$, $\varepsilon <\delta <\tau  $, $k,\ell, m \in\{1,2,3,4\}\backslash \{j\}$, $k<\ell <m$, and
$$
(\mbox {adj}_{3}\xi )^\gamma=\left( (\mbox {adj}_{3}\xi )^ \gamma _1 ,\cdots ,(\mbox {adj}_{3}\xi ) ^ \gamma_ 4 \right) \in \mathbb R^4, \  \ \gamma =1,2,3,4,
$$
is the $\gamma$th row of $\mbox {adj}_3 \xi \in \mathbb R^{4\times 4}$; moreover,
$\mbox {adj}_4\xi= \det \xi $ denotes the adjugate matrix of order 4, i.e., the determinant of the square matrix $\xi \in \mathbb R^{4\times 4}$.

\vspace{2mm}

We assume that there exist exponents $p\in (1,4]$, $q>1$, $r>1$, $s\ge 1$, constants $c_1, c_3>0$, $c_2\geq0$, and nonnegative functions
\begin{equation*}\label{conditions for abcd}
a(x),b(x),c(x),d(x) \in L_{loc}^\sigma (\Omega), \ \sigma >\frac 4 p,
\end{equation*}
such that for $\alpha \in\{1,2,3,4\}$, $\beta\in\{1,2,3,4,5,6\}$ and $\gamma \in\{1,2,3,4\}$,
\begin{equation}\label{2.3}
c_{1}|\lambda |^{p}-c_{2}\leq F_\alpha(x,\lambda )\leq c_{3} (|\lambda |^{p}+1) +a(x), \quad \forall \lambda\in \mathbb R^4,
\end{equation}
\begin{equation}\label{2.4}
c_{1}|\eta |^{q}-c_{2}\leq G_\beta(x,\eta)\leq c_{3} (|\eta|^{q}+1) +b(x), \quad \forall \eta\in \mathbb R^6,
\end{equation}
\begin{equation}\label{2.5}
c_{1}|\lambda |^{r}-c_{2}\leq H_\gamma (x,\lambda )\leq c_{3} (|\lambda |^{r}+1) +c(x),\quad \forall \lambda\in \mathbb R^4,
\end{equation}
\begin{equation}\label{2.6}
0\leq I(x,t)\leq c_{3} (|t|^{s}+1)+d(x), \quad \forall t\in \mathbb R.
\end{equation}

Note that we have assumed that the integrand $f(x,\xi)$ has splitting form, and the functions $F_\alpha(x,\cdot )$, $G_\beta(x,\cdot )$, $ H_\gamma (x,\cdot )$ and $I(x,\cdot )$ are convex, thus the function $f(x,\xi)$ defined in (\ref{2.2}) is polyconvex.
Recall that a function $f=f(\xi): \mathbb R^{m\times n} \rightarrow \mathbb R$ is said to be polyconvex if there exists a convex function $g: \mathbb R ^{\tau (m,n)} \rightarrow \mathbb R$ such that
$$
f(\xi) =g(T(\xi)),
$$
where
$$
\tau (m,n) =\sum_{i=1} ^{\min \{m,n\} } \left(\begin{split} &n\\ &i \end{split}\right)  \left(\begin{split} &m \\ &i \end{split}\right),
$$
and $T(\xi)$ is the vector defined as follows:
$$
T(\xi) =(\xi, {\rm adj}_2\xi, \cdots, {\rm adj}_i\xi, \cdots, {\rm adj}_{\min \{m,n\}}\xi),
$$
here ${\rm adj}_i\xi$ denotes the adjugate matrix of order $i$. In particular, if $m=n$, then ${\rm adj}_n \xi =\det \xi$.

\begin{definition}\label{definition}
A function $u\in W_{loc}^{1,1} (\Omega, \mathbb R^4)$ is a local minimizer of ${\cal F}(u,\Omega) $ in (\ref{integral functionals}) with $f(x,\xi)$ be as in (\ref{2.2}) if $f(x,Du (x)) \in L_{loc}^1(\Omega)$ and
$$
{\cal F} (u,{\rm supp} \varphi) \le {\cal F} (u+\varphi,{\rm supp} \varphi)
$$
for all $\varphi \in W^{1,1} (\Omega, \mathbb R^4)$ with ${\rm supp} \varphi \Subset \Omega$.
\end{definition}

The main result in this section is the following


\begin{theorem}\label{main theorem for section 2}
Let $f$ has splitting structure as in (\ref{2.2}), and satisfy the growth conditions (\ref{2.3})-(\ref{2.6}).
Let $u\in W^{1,1}_{loc} (\Omega,\mathbb R^4)$ be a local minimizer of $\mathcal{F}$ and $1\le s <r<q<p\le 4$. Assume
\begin{equation}\label{relation for theorem 2.2}
\frac {p}{p^*} <1- \max \left\{\frac {qp^*}{p(p^*-q)}, \frac {rp^*}{q(p^*-r)}, \frac {sp^*}{r(p^*-s)}, \frac 1 \sigma \right\},
\end{equation}
where $p^* =\frac {4p}{4-p}$, if $p<4$, and, if $p=4$, then $p^*$ is any $\nu >4$.

Then all the local minimizers $u\in W_{loc} ^{1,1} (\Omega, \mathbb R^{4})$ of $\cal F$ are locally bounded and locally H\"older continuous.
\end{theorem}

For related results to Theorem \ref{main theorem for section 2}, we should mention Cupini-Leonetti-Moscolo \cite{Cupini-Leonetti-Mascolo}, where the authors considered a special class of polyconvex functionals from $\Omega \subset \mathbb R^3$ to $\mathbb R^3$:
\begin{equation}\label{integrand from CLM}
f(x,\xi) =\sum _{\alpha =1}^3 \big \{F_\alpha (x,\xi^\alpha ) +G_\alpha (x, (\mbox {adj}_2 \xi )^\alpha )\big \} +H(x,\det \xi).
\end{equation}
Under some structural assumptions on the energy density, the authors proved that local minimizers $u$ are locally bounded by using De Giorgi's iteration method. Some related results can be found in \cite{Carozza-Gao-Giova-Leonetti,CFLM,Gao-Huang-Ren} and \cite{GSR},  where in \cite{Carozza-Gao-Giova-Leonetti} the authors found conditions on the structure of some integral functional which forces minimizers to be locally bounded; in \cite{CFLM}, the authors proved local H\"older continuity of vectorial local minimizers of special classes of integral functionals with rank-one and polyconvex integrands. The regularity of minimizers was obtained by proving that each component stays in a suitable De Giorgi class; in \cite{Gao-Huang-Ren} the authors gave regularity results for minimizers of two special cases of polyconvex functionals, under some structural assumptions on the energy density, the authors proved that the minimizers are either bounded, or have suitable integrability properties by using Stampacchia Lemma; and in \cite{GSR}, the authors obtained regularity properties for minimizing sequences of some integral functionals in 3-dimensional spaces related to nonlinear elasticity theory. Under some structural conditions on the energy density, the authors derived that the minimizing sequences and the derivatives of the sequences have some regularity properties by using the Ekeland variational principle. Partial regularity for minimizers of degenerate polyconvex energies can be found in \cite{ Esposito-Mingione }. Partial regularity results related to nonlinear elasticity theory can be found in \cite{ Fusco-Hutchinson,FH,Fuchs-Reuling, Fuchs-Seregin,Hamburger}. For some other contributions related to polyconvex functionals, we refer to \cite{Kristensen-Taheri,Szekelyhidi}.

\vspace{2mm}

We now present some preliminary lemmas that will be used in the proof of Theorem \ref{main theorem for section 2}.

\begin{lemma}\label{adj2}
Let $\xi \in \mathbb R^{4\times 4}$ be a square matrix, ${\rm adj}_2 \xi\in \mathbb R^{6\times 6}$ be the adjugate matrix of order 2 whose components are as in (\ref{components1}), ${\rm adj}_3 \xi\in \mathbb R^{4\times 4}$ be the adjugate matrix of order 3 whose components are as in (\ref{components 2}), and $\det \xi $ be the adjugate matrix of order 4 (that is, the determinant of $\xi$). Then

{\rm(i)}
$$
|({\rm adj}_2 \xi) ^\beta| \le 12  |\xi^1| \sum_{\alpha =2}^4|\xi ^\alpha |, \ \ \beta =1,2,3,
$$
where $({\rm adj}_2 \xi) ^\beta$ is the $\beta$th row of ${\rm adj}_2 \xi \in \mathbb R^{6\times 6}$;

{\rm(ii)}
$$
|({\rm adj}_3 \xi) ^\gamma| \le 12 |\xi ^1 | \sum _{\beta=4}^6 |({\rm adj}_2 \xi)^\beta |, \ \ \gamma =2,3,4,
$$
where $({\rm adj}_3 \xi) ^\gamma$ is the $\gamma $th row of ${\rm adj}_3 \xi \in \mathbb R^{4\times 4}$;

{\rm(iii)}
$$
|\det \xi| \le 4 |\xi ^1||({\rm adj}_3\xi )^1|.
$$
\end{lemma}

\begin{proof} From the definition of the adjugate matrix of order 2, we know that $({\rm adj}_2 \xi )^\beta$  is a 6-dimensional vector whose components are $({\rm adj}_2 \xi )^\beta _j$, $j=1,2,3,4,5,6$. For $\beta =1,2,3$, $({\rm adj}_2 \xi )^\beta _j$ is, up to a sign, a determinant of a $2\times 2$ matrix  with the 2 entries on the first row come from $\xi^1$, thus
$$
|({\rm adj}_2 \xi )^\beta_j| \le 2 |\xi^1| \sum_{k=2}^4|\xi ^k|, \ \ \beta =1,2,3, \   j=1,2,3,4,5,6,
$$
from which we derive
$$
|({\rm adj}_2 \xi) ^\beta| \le \sum _{j=1}^6 |({\rm adj}_2 \xi )^\beta_j| \le 12  |\xi^1| \sum_{k=2}^4|\xi ^k|, \ \beta =1,2,3.
$$


From the definition of the adjugate matrix of order 3, we know that $({\rm adj}_3 \xi )^\gamma$  is a 4-dimensional vector whose components are $({\rm adj}_3 \xi )^\gamma _j$, $j=1,2,3,4$. For $\gamma =2,3,4$, $({\rm adj}_3 \xi )^\gamma _j$ is, up to a sign, a determinant of a $3\times 3$ matrix with the 3 entries on the first row come from $\xi^1$. We use the cofactor expansion formula for $({\rm adj}_3 \xi )^\gamma _j$ and we have
$$
|({\rm adj}_3\xi )^\gamma_j| \le 3 |\xi ^1 | \sum _{\beta=4}^6 |({\rm adj}_2 \xi)^\beta |  , \ \ \gamma =2,3,4, \  j=1,2,3,4,
$$
from which we derive
$$
|({\rm adj}_3 \xi) ^\gamma| \le \sum _{j=1}^4 |({\rm adj}_2 \xi )^\gamma_j| \le 12 |\xi ^1 | \sum _{\beta=4}^6 |({\rm adj}_2 \xi)^\beta | , \ \gamma \in \{2,3,4\}.
$$


We use the cofactor expansion formula again to derive
$$
|\det \xi| \le \sum_{j=1}^4 |\xi ^1_j| |({\rm adj}_3\xi )^1_j| \le 4 |\xi ^1||({\rm adj}_3\xi )^1|.
$$
This ends the proof of Lemma \ref{adj2}.
\end{proof}

\begin{lemma}\label{lemma3.1}
Let $f$ be as in (\ref{2.2}) and satisfy the growth conditions (\ref{2.3})-(\ref{2.6}).
Let $u\in W^{1,1}_{loc} (\Omega,\mathbb R^4)$ be such that
$$
f(x,Du(x)) \in L^1_{loc}(\Omega).
$$
Fix $\eta \in C ^1_0(\Omega),\eta \geq0$ and $k\in \mathbb R$ and denote for almost every $x\in\{u^1>k\}\cap\{\eta>0\}$,
\begin{equation}\label{definition for A}
A =\left(
\begin{array}{ccccc}
p^* \eta^{-1}(k-u^1)D\eta\\
Du^2\\
Du^3\\
Du^4
\end{array}
\right),
\end{equation}
If
\begin{equation}\label{relation on q r s}
q< \frac {pp^* }{p+p^* }, \ r< \frac {qp^* }{q+p^* },  \ s< \frac {rp^* }{r+p^* },
\end{equation}
then
$$
\eta^{p^* } f(x,A)\in L^1\left(\{u^1>k\}\cap\{\eta>0\}\right).
$$
\end{lemma}


\begin{proof}  Denote
$$
\hat{u}: =\left(
\begin{array}{cccc}
0\\
u^2\\
u^3\\
u^4
\end{array}
\right), \   \
D\hat{u}: =\left(
\begin{array}{cccc}
0\\
Du^2\\
Du^3\\
Du^4
\end{array}
\right).
$$
By the definition of $f(x,\xi)$ in (\ref{2.2}) and the growth conditions (\ref{2.3})-(\ref{2.6}) we have, almost everywhere in $ \left\{u^{1}>k\right\} \cap\{\eta>0\} $,
\begin{equation}\label{3.02}
\begin{split}
|f(x,A)| \leq & c \Big\{ |A|^p  +\sum_{\beta=1}^{3}\left|\left(\operatorname{adj}_{2}A\right)^{\beta}\right|^{q} +\sum_{\beta=4}^{6}\left|\left(\operatorname{adj}_{2} A\right)^{\beta}\right|^{q}+\left| \left(\operatorname{adj}_{3} D u\right)^{1}\right|^{r} \\
& + \sum_{\gamma=2}^{4}\left|\left(\operatorname{adj}_{3} A\right)^{\gamma}\right|^{r} +|\det A|^s +\omega (x) \Big \},
\end{split}
\end{equation}
where $c$ is a constant depending upon $c_1,c_2,c_3$ and
$$
\omega(x)=a(x )+b(x) +c(x) +d(x)+1.
$$
It is obvious that, almost everywhere in $ \left\{u^{1}>k\right\} \cap\{\eta>0\} $,
\begin{equation}\label{3.1311}
|A|^p \le c \left( (p^* \eta ^{-1} (u^1-k) |D\eta|) ^p +|D\hat u | ^p\right),
\end{equation}
\begin{equation}\label{3.1312}
\left|\left({\rm adj}_{2} A\right)^{\beta}\right|^{q} =  \left|\left({\rm adj}_{2} Du\right)^{\beta}\right|^{q},\ \  \beta=4,5,6,
\end{equation}
Thanks to Lemma \ref{adj2}, we get a.e. in $ \left\{u^{1}>k\right\} \cap\{\eta>0\} $,
\begin{equation*}\label{2.14}
\begin{split}
&\left|\left(\operatorname{adj}_{2} A\right)^{\beta}\right|^{q} \leq 12^q \Big(p^* \eta^{-1}(u^1-k)|D\eta| \sum_{\alpha =2}^4 |Du^\alpha| \Big) ^{q}, \ \ \beta =1,2,3, \\
& \left|\left(\operatorname{adj}_{3} A\right)^{\gamma}\right|^{r} \leq 12^r \Big(p^* \eta^{-1}(u^1-k)  |D\eta| \sum_{\beta=4}^6 |({\rm adj}_2 Du )^\beta |\Big) ^{r}, \ \ \gamma =2,3,4,\\
& |\det A|^s \le 4^s \left( (p^* \eta^ {-1} (u^1-k)|D\eta|) |({\rm adj}_3 Du)^1| \right)^s.
\end{split}
\end{equation*}
We use Young inequality and we derive from the above inequalities that, almost everywhere in $ \left\{u^{1}>k\right\} \cap\{\eta>0\}$,
\begin{equation*}\label{3.03}
\begin{split}
& \sum_{\beta=1}^{3}\left|\left(\operatorname{adj}_{2} A\right)^{\beta}\right|^{q}+\sum_{\gamma=2}^{4}\left|\left(\operatorname{adj}_{3} A\right)^{\gamma}\right|^{r} +|\det A|^s\\
\leq & c \Big\{\Big(p^* \eta^{-1}(u^1-k)|D\eta| \sum_{\alpha =2}^4 |Du^\alpha| \Big) ^{q}+\Big(p^* \eta^{-1}(u^1-k)  |D\eta| \sum_{\beta=4}^6 |({\rm adj}_2 Du )^\beta |\Big) ^{r} \\
    &  + \left( (p^* \eta^ {-1} (u^1-k)|D\eta|) |({\rm adj}_3 Du)^1| \right)^s \Big\}\\
\leq & c \Big \{ (p^* \eta^{-1}(u^1-k)|D\eta|) ^{\frac {pq} {p-q}}+ \sum_{\alpha =2}^4 |Du^\alpha|^{p}   + (p^* \eta^{-1}(u^1-k)  |D\eta|)^{\frac {qr}{q-r}}
\end{split}
\end{equation*}
\begin{equation}\label{3.03}
\begin{split}
& +\sum_{\beta=4}^6 |({\rm adj}_2 Du )^\beta |^q  + (p^* \eta^ {-1} (u^1-k)|D\eta|) ^{\frac {rs}{r-s}}+ |({\rm adj}_3 Du)^1|^r \Big\},
\end{split}
\end{equation}
where $c$ is a constant depending upon $p,q,r,s$. Denote
$$
\tilde{q}:=\max \left\{\frac {pq} {p-q}, \frac {qr} {q-r}, \frac {rs}{r-s}\right\},
$$
then
\begin{equation*}
\begin{array}{llll}
&\displaystyle  (p^* \eta^{-1}(u^1-k)|D\eta|) ^{\frac {pq} {p-q}}+ (p^* \eta^{-1}(u^1-k)  |D\eta|)^{\frac {qr}{q-r}}+ (p^* \eta^ {-1} (u^1-k)|D\eta|) ^{\frac {rs}{r-s}} \\
\le &\displaystyle c (p^* \eta^{-1}(u^1-k)|D\eta|) ^{\tilde q } +1),
\end{array}
\end{equation*}
with $c=c(p,q,r,s)$ a positive constant. (\ref{relation on q r s}) ensures $\tilde q <p^*$.
The above inequality together with (\ref{3.03}) implies
\begin{equation}\label{2.16}
\begin{split}
& \sum_{\beta=1}^{3}\left|\left(\operatorname{adj}_{2} A\right)^{\beta}\right|^{q}+\sum_{\gamma=2}^{4}\left|\left(\operatorname{adj}_{3} A\right)^{\gamma}\right|^{r} +|\det A|^s\\
\leq & c \Big \{ (p^* \eta^{-1}(u^1-k)|D\eta|) ^{\tilde q }+ \sum_{\alpha =2}^4 |Du^\alpha|^{p} +\sum_{\beta=4}^6 |({\rm adj}_2 Du )^\beta |^q + |({\rm adj}_3 Du)^1|^r +1 \Big\} .
\end{split}
\end{equation}
Substituting (\ref{3.1311}), (\ref{3.1312}) and (\ref{2.16}) into (\ref{3.02}), one has, almost everywhere in $ \left\{u^{1}>k\right\} \cap\{\eta>0\} $,
\begin{equation}\label{3.1911}
\begin{split}
\eta^{p^* } |f(x,A)| \leq & c\Big\{\eta ^{p^*-p} (p^* (u^1-k)|D\eta |)^p +\eta ^{p^*} |D\hat u| ^p \\
& +(p^*) ^{\tilde{q}} \eta^{p^*  -\tilde{q}}\left(u^{1}-k\right)^{\tilde{q}}|D \eta|^{\tilde{q}}+\eta ^{p^* } \sum_{\alpha=2}^4 |Du^\alpha| ^p \\
&  +\eta^{p^* }\sum_{\beta =4}^6 |({\rm adj}_2 Du) ^\beta |^q +\eta^{p^*}|\left(\operatorname{adj}_{3} D u\right)^1|^{r}+\eta^{p^*} \omega(x)\Big\}.
\end{split}
\end{equation}
By the growth conditions (\ref{2.3})-(\ref{2.6}) and the fact $ f(x,Du(x)) \in L_{\text {loc }}^{1}(\Omega)$  we obtain
\begin{equation}\label{3.2011}
\begin{split}
& \eta ^{p^*} |D\hat u|^p+  \eta ^{p^*}\sum_{\alpha=2}^4 |Du^\alpha| ^p  +\eta^{p^*}\sum_{\beta =4}^6 |({\rm adj}_2 Du) ^\beta |^q +\eta^{p^*}|\left(\operatorname{adj}_{3} D u\right)^1|^{r}\\
\leq & c \eta ^{p^*} (f(x,Du)+1) \in L^1 (\{u^1>k\} \cap \{\eta >0\}) .
\end{split}
\end{equation}
The definition for $f(x,\xi)$ in (\ref{integral functionals}), the conditions (\ref{2.3})-(\ref{2.6}) and $f(x,Du(x)) \in L_{loc} ^1(\Omega)$ imply $u\in W_{loc} ^{1,p} (\Omega)$, thus $ u \in L_{\text {loc }}^{p^{*}}\left(\Omega ; \mathbb{R}^{4}\right)$ by Sobolev Embedding Theorem. Since $p^* -\tilde q>0 $, we then have
\begin{equation}\label{3.2111}
\begin{array}{llll}
& \displaystyle \eta ^{p^*-p} (p^* (u^1-k)|D\eta |)^p+ (p^*)^{\tilde{q}} \eta^{p^*-\tilde{q}}\left(u^{1}-k\right)^{\tilde{q}}|D \eta|^{\tilde{q}}+\eta^{p^*} \omega(x)\\
 & \displaystyle  \in L^{1}\left(\left\{u^{1}>k\right\} \cap\{\eta>0\}\right) .
\end{array}
\end{equation}
(\ref{3.1911}) together with (\ref{3.2011}) and (\ref{3.2111}) implies $\eta^{p^*}|f(x,A)| \in L^{1}\left(\{u^{1}>k\}\cap \{\eta>0 \}\right)$, and the proof of Lemma \ref{lemma3.1} has been finished.
\end{proof}

For a local minimizer $ u=\left(u^{1}, u^{2}, u^{3},u^{4}\right)^t $ of ${\cal F} (u,\Omega)$ in (\ref{integral functionals}), the particular structure (\ref{2.2}) of the variational integral $f (x,Du)$ guarantee an extension of De Giorgi Class for any component $u^\alpha$ of $u$ on every superlevel set $\{u^\alpha>k\}$ and every sublevel set $\{u^\alpha< k\}$, and we then use Theorem \ref{theorem 21} to derive that it is locally bounded and locally H\"older continuous.
In the following we consider the first component $u^1$ (we can argue similarly for the other components  $u^{2}$, $ u^{3}$, $ u^{4}$).

\vspace{3mm}


\begin{lemma}\label{proposition 1}
Let f be as in (\ref{2.2}) satisfying the growth conditions (\ref{2.3})-(\ref{2.6}).
Suppose $u\in W^{1,1}_{loc} (\Omega,\mathbb R^4)$ be a local minimizer of $\mathcal{F}$. Let $1<p\le 4$ and $q,r,s$ satisfy the relations (\ref{relation on q r s}). Let $B_{R_0} (x_0)\Subset \Omega$ with $|B_{R_0}|<1$.
For $k\in \mathbb R$, denote
$$
{A^1_{k, t}}:=\{x\in B_t(x_0):u^1(x)>k\}, \quad 0<t\leq R_0 .
$$
Then there exists $c>0$, independent of $k$, such that for every $0< \rho <R\leq R_0$ and every $p\leq Q<p^*$ (where we recall $p^*=\frac {4p}{4-p}$ if $p<4$, and $p^*=$ any $\nu >p$ for $p=4$),
\begin{equation}\label{3.1}
\int_{A^{1}_{k, \rho}}|D u^{1}|^{p} dx \leq  \int_{A^{1}_{k, R}} \Big( \frac{u^1-k}{R-\rho} \Big)^{Q} d x +c |A^1_{k, R}|^{\theta},
\end{equation}
where
\begin{equation}\label{definition for theta}
\theta =1- \max \left\{\frac{q Q}{p(Q-q)}, \frac {Qr}  {q(Q-r)} , \frac{Qs}{r(Q-s)},\frac 1 \sigma \right\}.
\end{equation}
\end{lemma}


\begin{proof} Let  $B_{R_0} (x_0) \Subset \Omega$, $\left|B_{R_0}\right|<1 $ (which obviously implies $ R_0<1$). Let  $\rho,R $  be such that $0< \rho<R \leq R_0 $. Consider a cut-off function  $ \eta \in C_{0}^{\infty}\left(B_{R}\right)$  satisfying the following assumptions
\begin{equation}\label{3.2}
0 \leq \eta \leq 1, \  \eta \equiv 1   \mbox { in }   B_{\rho}\left(x_{0}\right),  \ |D \eta| \leq \frac{2}{R-\rho} .
\end{equation}
Fixing  $k \in \mathbb{R} $, define  $w \in W_{\text {loc }}^{1,1}\left(\Omega ; \mathbb{R}^{4}\right) $ as
$$
w^{1}:=(u^1-k)\vee 0=\max \left\{u^{1}-k, 0\right\}, \quad w^{2}:=0, \quad w^{3}:=0,\quad w^{4}:=0,
$$
and
$$
\varphi:=-\eta^{p^*}  w.
$$
We have $\varphi=0$ a.e in $ \Omega \backslash \left(\{\eta>0\} \cap\left\{u^{1}>k\right\}\right) $, thus
\begin{equation}\label{3.3}
 f(x, Du+D \varphi)=f(x, Du)  \mbox { almost everywhere in } \;  \Omega \backslash \left(\{\eta>0\} \cap\left\{u^{1}>k\right\}\right).
\end{equation}
Denote $A$ by (\ref{definition for A}). For almost every $x\in \{\eta>0\} \cap\left\{u^{1}>k\right\}$,
we notice that
\begin{equation*}
Du+D \varphi=\left(\begin{array}{c}
\left(1-\eta^{p^*}\right) D u^{1}+{p^*}  \eta^{{p^*}-1}\left(k-u^{1}\right) D \eta \\
D u^{2} \\
D u^{3} \\
D u^{4}
\end{array}\right)=\left(1-\eta^{p^*}\right) D u+\eta^{p^*} A.
\end{equation*}
Since
\begin{equation*}
\operatorname{adj}_{2}(D u+D \varphi)=\left(1-\eta^{p^*}\right) \operatorname{adj}_{2} D u+\eta^{p^*} \operatorname{adj}_{2} A,
\end{equation*}
\begin{equation*}
\operatorname{adj}_{3}(D u+D \varphi)=\left(1-\eta^{p^*}\right) \operatorname{adj}_{3} D u+\eta^{p^*} \operatorname{adj}_{3} A
\end{equation*}
and
\begin{equation*}
\operatorname{det}(D u+D \varphi)=\left(1-\eta^{p^*}\right) \operatorname{det} D u+\eta^{p^*} \operatorname{det} A,
\end{equation*}
then thanks to the assumptions that $F_\alpha (x,\cdot)$, $G_\beta (x,\cdot)$, $H_\gamma (x,\cdot)$ and $I (x,\cdot)$ are convex, we get $f(x,\cdot)$ in (\ref{2.2}) is polyconvex, thus almost everywhere in $\{\eta>0\} \cap \{u^{1}>k \}$,
\begin{equation}\label{3.5}
\begin{array}{llll}
f(x, D u+D \varphi)&\displaystyle =f(x,(1-\eta^{p^*}) D u+\eta^{p^*} A)\\
&\displaystyle \leq (1-\eta^{p^*}) f(x, D u)+\eta^{p^*} f(x,A).
\end{array}
\end{equation}
By the minimality of $u$, $f(x, Du) \in L_{loc}^{1}(\Omega) $. Lemma \ref{lemma3.1} ensures that
\begin{equation*}
\eta^{p^*} f(x, A) \in L^{1}\left(\left\{u^{1}>k\right\} \cap\{\eta>0\}\right),
\end{equation*}
therefore (\ref{3.3}) and (\ref {3.5}) imply $f(x, D u+D \varphi) \in L_{\text {loc }}^{1}(\Omega)$, then
\begin{equation*}
\int_{A_{k, R}^{1} \cap\{\eta>0\}} f(x,D u) \mathrm{d} x \leq \int_{A_{k, R}^{1} \cap\{\eta>0\}} \left\{(1-\eta^{p^*}) f(x, D u)+\eta^{p^*}f(x,A)\right\} \mathrm{d} x.
\end{equation*}
The inequality above implies
\begin{equation}\label{3.6}
\int_{A_{k, R}^{1} \cap\{\eta>0\}} \eta^{p^*} f(x,D u) \mathrm{d} x \leq \int_{A_{k, R}^{1} \cap\{\eta>0\}} \eta^{p^*} f(x,A) \mathrm{d} x.
\end{equation}
Taking into account $A$ in (\ref{definition for A}) and the particular structure of  $f$ in (\ref{2.2}) we obtain
\begin{equation*}\label{3.7}
\begin{split}
F_{2}\big(x,A^{2}\big)&=F_{2}\big(x,(D u)^{2}\big), \\
F_{3}\big(x,A^{3}\big)& =F_{3}\big(x,(D u)^{3}\big), \\
F_{4}\big(x,A^{4}\big)&=F_{4}\big(x,(D u)^{4}\big)\\
G_{4}\big(x,\big(\operatorname{adj}_{2} A\big)^{4}\big)&=G_{4} \big(x,\big(\operatorname{adj}_{2} D u\big)^{4}\big),\\
G_{5}\big(x,\big(\operatorname{adj}_{2} A\big)^{5}\big) & =G_{5}\big(x,\big(\operatorname{adj}_{2} D u\big)^{5}\big), \\
G_{6}\big(x,\big(\operatorname{adj}_{2} A\big)^{6}\big) & =G_{6}\big(x,\big(\operatorname{adj}_{2} D u\big)^{6}\big),\\
H_{1}\big(x,\big(\operatorname{adj}_{3} A\big)^{1}\big)& =H_{1}\big(x,\big(\operatorname{adj}_{3} D u\big)^{1}\big),
\end{split}
\end{equation*}
then by (\ref{3.6}) and the growth assumptions (\ref{2.3})-(\ref{2.6}) one has
\begin{equation}\label{3.8}
\begin{split}
&\int_{A_{k, R}^{1} \cap\{\eta>0\}} \eta^{p^*} \Big \{|Du^1|^p -1 \Big \}\mathrm{d}x \\
\le & \int_{A_{k, R}^{1} \cap\{\eta>0\}} \eta^{p^*} \Big \{|Du^1|^p +\sum_{\beta =1}^3 |(\mbox {adj}_2 Du) ^\beta|^q
 +\sum_{\alpha =2}^4 |(\mbox {adj}_3 Du) ^\alpha| ^r -1 \Big \}\\
\le & \displaystyle c\int_{A_{k, R}^{1} \cap\{\eta>0\}} \eta^{p^*}\Big\{ F_{1}\left(x,D u^{1}\right)+\sum_{\beta=1}^{3} G_{\beta} \left(x,\left(\operatorname{adj}_{2} Du\right)^{\beta}\right) \\
& \displaystyle  \qquad \qquad  +\sum_{\gamma=2}^{4} H_{\gamma}\left(x,\left(\operatorname{adj}_{3} Du\right)^{\gamma }\right) +I(x,\operatorname{det} D u)\Big \} \mathrm{d}x\\
\leq & \displaystyle c\int_{A_{k, R}^{1} \cap\{\eta>0\}} \eta^{p^*}\Big \{ F_{1}\left(x, p^*\eta^{-1}\left(k-u^{1}\right) D \eta\right)+\sum_{\beta=1}^{3} G_{\beta}\big(
    x,\left(\operatorname{adj}_{2} A\right)^{\beta}\big)\\
& \displaystyle  \qquad \qquad +\sum_{\gamma =2}^{4} H_{\gamma }\left(x,\left(\operatorname{adj}_{3} A\right)^{\gamma }\right)+I(x,\operatorname{det} A)\Big\} \mathrm{d}x,
\end{split}
\end{equation}
where $c$ is a constant depending upon $c_1,c_2,c_3$.

Our nearest goal is to estimate the right hand side integrals.
By the assumptions (\ref{2.3}) and (\ref{3.2}), we use Young inequality in order to derive, for any  $Q \in (p,p^*)$,
\begin{equation}\label{3.10}
\begin{split}
& \int_{A_{k, R}^{1} \cap\{\eta>0\}} \eta^{p^*} F_{1}\left(x, p^*\eta^{-1}\left(k-u^{1}\right) D \eta\right) \mathrm{d} x\\
\le & c\int_{A_{k, R}^{1} \cap\{\eta>0\}} \eta^{p^*}\left(|p^*\eta^{-1}\left(u^{1}-k \right) D \eta| ^p +a(x)+1 \right)   \mathrm{d} x\\
\le &c \int_{A_{k, R}^{1} \cap\{\eta>0\}}  \Big((p^*)^p \eta^{p^*-p}\left(\frac{u^{1}-k}{R-\rho}\right)^{p}+ a(x)+1 \Big) \mathrm{d} x\\
\leq &c \int_{A_{k, R}^{1}} \left( \frac{u^{1}-k}{R-\rho}\right)^{p}\mathrm{d} x +c \int_{A_{k, R}^{1}} (a(x)+1) \mathrm{d} x\\
\leq &  c \int_{A_{k, R}^{1}}\left(\frac{u^{1}-k}{R-\rho}\right)^{Q} \mathrm{d} x +c|A^{1}_{k,R}| +c \int_{A_{k, R}^{1}} (a(x)+1) \mathrm{d} x,
\end{split}
\end{equation}
here $c$ is a constant depending on $c_3$, $ p$, and we have used the facts $\eta ^{p^*-p}\le 1$ and
$$
|A_{k,R}^1 |=|A_{k,R}^1 | ^{\frac 1 \sigma }|A_{k,R}^1 | ^{1-\frac 1 \sigma }\le |\Omega |^{\frac 1 \sigma }|A_{k,R}^1 | ^{1-\frac 1 \sigma }.
$$

\vspace{2mm}

Recall $D\hat{u}=(0, Du_2,Du_3,Du_4)^t$.
We use the growth condition in (\ref{2.4}) for $G_{\beta}$ and Lemma \ref{adj2} (i) in order to derive
\begin{equation*}
\begin{array}{l}
\begin{split}
& \eta^{p^*} \sum_{\beta=1}^{3} G_{\beta}\big( x,\big(\operatorname{adj}_{2} A\big)^{\beta}\big ) \\
\leq  & c \eta ^{p^*} \sum_{\beta=1}^{3}\big(\big|\big(\operatorname{adj}_{2} A\big)^{\beta}\big |^{q}+b(x)+1 \big) \\
\leq  & c \eta^{p^*}  \sum_{\beta=1}^{3}\left( |p^* \eta ^{-1} (u^1-k ) D\eta |^{q} \left(\sum_{\alpha =2} ^4 |Du^\alpha|\right)^q  +b(x)+1 \right) \\
\leq  & c \eta ^{p^*} \sum_{\beta=1}^{3} |p^* \eta ^{-1} (u^1-k ) D\eta |^{q} |D\hat{u}|^q+b(x)+1 \big) \\
\leq & c \eta^{p^* -q}(p^* )^{q} \left(\frac{u^{1}-k}{R-\rho}\right)^{q} |D\hat{u}|^q+c\eta^{p^*} (b(x)+1),
\end{split}
\end{array}
\end{equation*}
where $c$ is a constant depending on $c_3, q$.
From (\ref{relation on q r s}) we know  $q <\frac {pp^*}{p+p^*} <p<Q$, this  allows us to use the Young inequality with exponents $\frac{Q}{q}$ and $\frac{Q}{Q-q}$ in order to derive, for almost every $x\in {A_{k, R}^{1} \cap\{\eta>0\}}$,
$$
\left(\frac{u^{1}-k}{R-\rho}\right)^{q}|D\hat{u}|^q \leq \left(\frac{u^{1}-k}{R-\rho}\right)^{Q}+|D\hat{u}|^{\frac{Qq}{Q-q}}.
$$
We have thus derived that for any $Q \in (p,p^*)$,
\begin{equation}\label{3.11}
\begin{array}{l}
\begin{split}
&\int_{A_{k, R}^{1} \cap\{\eta>0\}} \eta^{p^*} \sum_{\beta=1}^{3} G_{\beta}\big( x,\big(\operatorname{adj}_{2} A\big )^{\beta}\big ) \mathrm{d} x \\
\le & c \int_{A_{k, R}^{1}}\Big \{ \Big (\frac{u^{1}-k}{R-\rho}\Big )^{Q}+|D\hat{u}|^{\frac{Qq}{Q-q}} +b(x)+1 \Big \},
\end{split}
\end{array}
\end{equation}
with $c$ a constant depending upon $c_2, c_3,p,Q, q$.

By the growth condition (\ref{2.5}) for $H_\gamma $ and Lemma \ref{adj2} (ii) we get
\begin{equation*}
\begin{array}{l}
\begin{split}
& \eta^{p^*} \sum_{\gamma =2}^{4} H_{\gamma }\left(x,\left(\operatorname{adj}_{3} A\right)^{\gamma }\right) \\
\leq &  c \eta^{p^*} \sum_{\gamma =2}^{4}\left(\left|\left(\operatorname{adj}_{3} A\right)^{\gamma }\right|^{r}+c(x)+1 \right) \\
\leq & c \eta^{p^*} \sum_{\gamma =2}^{4}\left(\Big|p^*\eta ^{-1} (u^1-k) |D\eta | \sum_{\beta=4}^6 |({\rm adj}_2 A )^\beta|\Big|^{r}+c(x)+1 \right) \\
\leq & c\eta^{p^* -r} (p^*)^{r} \Big (\frac{u^{1}-k}{R-\rho}\Big)^{r} \left( \sum_{\beta=4}^6 |({\rm adj}_2 D \hat u)^\beta|\right)^{r}+c\eta^{p^*} (c(x)+1).
\end{split}
\end{array}
\end{equation*}
From (\ref{relation on q r s}) we know that $r<\frac {qp^*}{q+p^*}<q <p<Q$, which allows us to use Young inequality with exponents $\frac{Q}{r}$ and $\frac{Q}{Q-r}$ to get almost everywhere in ${A_{k, R}^{1} \cap\{\eta>0\}}$,
$$
\Big (\frac{u^{1}-k}{R-\rho}\Big)^{r} \left( \sum_{\beta=4}^6 |({\rm adj}_2 D \hat u)^\beta|\right)^{r} \leq c \left(\left(\frac{u^{1}-k}{R-\rho}\right)^{Q}+ \left( \sum_{\beta=4}^6 |({\rm adj}_2 D\hat u)^\beta|\right) ^{\frac{Qr}{Q-r}} \right).
$$
The above inequality, together with the fact $\eta ^{p^*-r} \le 1$ implies
\begin{equation}\label{3.12}
\begin{array}{l}
\begin{split}
&\int_{A_{k, R}^{1} \cap\{\eta>0\}} \eta^{p^*} \sum_{\gamma=2}^{4} H_{\gamma }\left(x,\left(\operatorname{adj}_{3} A\right)^{\gamma}\right) \mathrm{d} x\\
\le & c \int_{A_{k, R}^{1}} \left( \left(\frac{u^{1}-k}{R-\rho}\right)^{Q}+ \left( \sum_{\beta=4}^6 |({\rm adj}_2 D\hat u)^\beta|\right)^{\frac{Qr}{Q-r}}  +c(x)+1 \right)  \mathrm{d} x,
\end{split}
\end{array}
\end{equation}
with $c$ a constant depending upon $c_2, c_3,p,Q,r$.

By (\ref{2.6}) and Lemma \ref{adj2} (iii),
$$
\begin{array}{llll}
&\displaystyle \eta^{p^*} I(x,\det A) \\
\leq &\displaystyle  c\eta^{p^*}(|\det A|^s+ d(x) +1) \\
\le &\displaystyle c\eta^{p^*}\left( (|p^*\eta ^{-1} (u^1-k) |D\eta |  |({\rm adj}_3 A)^1|)^s +d(x)+1 \right)\\
\le &\displaystyle c\eta^{p^*-s }(p^*)^{s}\left(\frac{u^{1}-k}{R-\rho}\right)^{s}  |({\rm adj}_3 D\hat u)^1|^s+c\eta^{p^*} ( d(x)+1) .
\end{array}
$$
From (\ref{relation on q r s}) we know that $s<\frac {rp^*}{r+p^*} <r<p<Q$, which allows us to use Young inequality with exponents $\frac Qs$ and $\frac{Q}{Q-s}$ and get almost everywhere in ${A_{k, R}^{1} \cap\{\eta>0\}}$,
$$
 \left(\frac{u^{1}-k}{R-\rho}\right)^{s}|({\rm adj}_3 Du)^1|^s \leq c \left( \left(\frac{u^{1}-k}{R-\rho}\right)^{Q}+|({\rm adj}_3 D\hat u)^1| ^{\frac{Qs}{Q-s}} \right) .
$$
Therefore
\begin{equation}\label{3.13}
\begin{split}
&\displaystyle \int_{A_{k, R}^{1} \cap\{\eta>0\}} \eta^{p^*} I(x,\det A) \mathrm{d} x \\
\leq &\displaystyle  c \int_{A_{k, R}^{1}}\Big \{\Big (\frac{u^{1}-k}{R-\rho}\Big )^{Q}+ |({\rm adj}_3 D\hat u)^1| ^{\frac{Qs}{Q-s}}+d(x) +1 \Big \} \mathrm{d} x,
\end{split}
\end{equation}
with $c$ a constant depending upon $c_2, c_3,p,Q,s$.

Substituting the estimates (\ref{3.10}), (\ref{3.11}), (\ref{3.12}) and (\ref{3.13}) into (\ref{3.8}) we arrive at
\begin{equation}\label{3.14}
\begin{split}
& \int_{A_{k, \rho}^{1}}|Du^1|^p\mathrm{d}x
\leq \int_{A_{k, R}^{1} \cap \{\eta >0\}} \eta ^{p^*}  |Du^1|^p \mathrm{d}x \\
\le & c \int_{A_{k, R}^{1}} \left[\left(\frac{u^{1}-k}{R-\rho}\right)^{Q} +|D\hat{u}|^{\frac{Qq}{Q-q}}+ \left( \sum_{\beta=4}^6 |({\rm adj}_2 D\hat u)^\beta|\right)^{\frac{Qr}{Q-r}} \right. \\
 & \qquad \qquad \left. + |({\rm adj}_3 D \hat u)^1| ^{\frac{Qs}{Q-s}}+\omega(x)\right] \mathrm{d} x +c|A^{1}_{k,R}|^{1-\frac 1 \sigma },
\end{split}
\end{equation}
where $c$ is a constant depending upon $c_1,c_2,c_3,p,q,r,s,Q$ and
$$
\omega (x) =a(x) +b(x) +c(x) +d(x)+1.
$$

Considering the facts $Du\in L_{loc}^{p} (\Omega)$, ${\rm adj}_2 Du \in L_{loc}^q (\Omega, \mathbb R^{6\times 6})$ and ${\rm adj}_3 Du \in L_{loc}^r(\Omega, \mathbb R^{4\times 4})$ (which can be derived by using (\ref{2.3})-(\ref{2.5}) and $f(x,Du) \in L_{loc} ^1(\Omega)$), one can use H\"older inequality to derive
\begin{equation}\label{3.141}
\begin{split}
& \int_{A_{k,R}^{1}}\left[|D\hat{u}|^{\frac{Qq}{Q-q}}+ \left( \sum_{\beta=4}^6 |({\rm adj}_2 D\hat u)^\beta|\right)^{\frac{Qr}{Q-r}} + |({\rm adj}_3 D\hat u)^1| ^{\frac{Qs}{Q-s}}+\omega(x)\right]   \mathrm{d} x\\
\leq & \left(\int_{B_{R_0}}\left(D\hat{u}\right)^{p} \mathrm {d} x\right)^{\frac{qQ}{\left(Q-q\right) p}} \left|A_{k, R}^{1}\right|^{1-\frac{q Q}{p\left(Q-q\right)}} \\
& + \left(\int_{B_{R_0}}  \left( \sum_{\beta=4}^6 |({\rm adj}_2 D\hat u)^\beta|\right)^{q}  \mathrm {d} x\right)^{\frac {Qr}  {q(Q-r)} } \left|A_{k, R}^{1}\right|^{1-\frac {Qr}  {q(Q-r)} }  \\
& + \left(\int_{B_{R_0}} \left(|({\rm adj}_3 D\hat u)^1| \right)^{r}  \mathrm {d} x\right)^{\frac{Qs}{r(Q-s)}} \left|A_{k, R}^{1}\right|^{1- \frac{Qs}{r(Q-s)} }  \\
& + \|\omega (x)\| _{L^\sigma (B_{R_0})} |A_{k,R_0}^1 | ^{1-\frac 1 \sigma}\\
\leq & c\left(\left|A_{k, R}^{1}\right|^{1-\frac{q Q}{p\left(Q-q\right)}} +\left|A_{k, R}^{1}\right|^{1-\frac {Qr}  {q(Q-r)} }+ \left|A_{k, R}^{1}\right|^{1- \frac{Qs}{r(Q-s)} }+ |A_{k,R_0}^1 | ^{1-\frac 1 \sigma} \right).
\end{split}
\end{equation}
We stress that, the above $c$ is a constant not only depending on $c_1,c_2,c_3, p,q,r,s, Q$, but also on $\|D\hat u\|_{L^p(B_{R_0})}$ and $\|\omega\| _{L^\sigma (B_{R_0})}$, but independent of $k$.
Let us take $\theta$ as in (\ref{definition for theta})
and using the fact $\left|A_{k, R}^{1}\right|\le \left|B_{R_{0}}\right|<1$, then (\ref{3.14}) and (\ref{3.141}) merge into
\begin{equation*}\label{3.16}
\int_{A^{1}_{k, \rho}}|D u^{1}|^{p} dx \leq  c \int_{A^{1} _ {k, R}}\Big ( \frac{u^1-k}{R-\rho} \Big)^{Q} d x +c |A^1_{k, R}|^{\theta},
\end{equation*}
as desired.
\end{proof}

\begin{lemma}\label{proposition 2}
Let f be as in (\ref{2.2}) satisfying the growth conditions (\ref{2.3})-(\ref{2.6}).
Let $u\in W^{1,1}_{loc} (\Omega,\mathbb R^4)$ be a local minimizer of $\mathcal{F}$. Let $1<p\le 4$ and $q,r,s$ satisfy the relations (\ref{relation on q r s}). Let $B_{R_0} (x_0)\Subset \Omega$ with $|B_{R_0}|<1$,
For $k\in \mathbb R$, denote
\begin{equation}\label{definition for B}
{B^1_{k, t}}:=\{x\in B_t(x_0):u^1(x)<k\}, \quad 0<t\leq R_0 .
\end{equation}
Then there exists $c>0$, independent of $k$, such that for every $0< \rho <R\leq R_0$ and every $p<Q<p^*$,
\begin{equation}\label{3.111}
\int_{B^{1}_{k, \rho}}|D u^{1}|^{p} dx \leq  \int_{B ^{1}_{k, R}} \Big( \frac{k-u^1}{R-\rho} \Big)^{Q} d x +c |B^1_{k, R}|^{\theta},
\end{equation}
where $\theta $ is defined by (\ref{definition for theta}).
\end{lemma}

\begin{proof}
Denote $B^1_{k,t}$ as in (\ref{definition for B}). To prove that $u^{1}$ satisfies (\ref{3.111}), we notice that  $-u$  is a local minimizer of $ \int_{\Omega} \tilde{f}(x,Dv(x)) \mathrm{d} x $ where
\begin{equation*}
\tilde{f}(\xi):=\sum _{\alpha =1} ^4 F_\alpha (x,-\xi ^\alpha)+ \sum_{\beta=1}^6 G_\beta(x,(\mbox {adj}_{2}\xi)^\beta)+\sum_{\gamma =1}^4 H_\gamma(x,(- \mbox {adj}_{3}\xi)^\gamma )+H(x, \det \xi).
\end{equation*}
If we denote
$$
\tilde{F}_{\alpha}(x,\lambda):=F_{\alpha}(x,-\lambda), \quad \tilde{H}_{\gamma}(x,\lambda)=H_{\gamma}(x,-\lambda), \quad \lambda \in \mathbb{R}^{4},
$$
then the functions $ \tilde{F}_{\alpha}$ , $\tilde{H}_{\gamma}$  are convex and satisfy (\ref{2.3}) and (\ref{2.5}) respectively. The function $ \tilde{f} $  satisfies the assumptions of Lemma \ref{proposition 1}, therefore we have obtained that $v^1 =-u^1$ satisfies (\ref{3.1}), which is equivalent to $u^1$ satisfies (\ref{3.111}).
\end{proof}

\begin{rem}
The symmetric structure of the energy density $f(x,\xi)$ in (\ref{2.2}) allows us to obtain analogous statements to Lemmas \ref{proposition 1} and \ref{proposition 2}  also for  $u^{2}$, $u^{3}$ and $u^{4}$.
\end{rem}

With the above lemmas in hands, we are now ready to prove Theorem \ref{main theorem for section 2}.

\begin{proof}
Under the condition (\ref{relation for theorem 2.2}), one can choose $Q\in (p,p^*)$ sufficiently close to $p^*$ such that
\begin{equation}\label{inequality no.1}
\frac {p}{p^*} <  1-\frac {q Q}{p(Q-q)}<  1-\frac {qp^*}{p(p^*-q)},
\end{equation}
\begin{equation}\label{inequality no.2}
\frac {p}{p^*} < 1-\frac {rQ}{q(Q-r)}<1-\frac {rp^*}{q(p^*-r)}
\end{equation}
\begin{equation}\label{inequality no.3}
\frac {p}{p^*} < 1-\frac {sQ}{r(Q-s)}<1-\frac {sp^*}{r(p^*-s)}.
\end{equation}
These inequalities can always be satisfied since the right hand side inequalities of the above three relations are all equivalent to $Q<p^*$. (\ref{inequality no.1}), (\ref{inequality no.2}), (\ref{inequality no.3})  together with $\sigma >\frac 4 p  $ imply
$$
\frac {p}{p^*} <1- \max \left\{\frac {q Q}{p(Q-q)}, \frac {rQ}{q(Q-r)}, \frac {sQ}{r(Q-s)}, \frac 1 \sigma \right\} .
$$
We recall the definition of $\theta $ in (\ref{definition for theta}) and we have
$$
\theta >\frac p{p^*} =1-\frac p 4.
$$
The result (\ref{3.1}) in Lemma \ref{proposition 1} tells us that $u^1$ satisfies (\ref{eDG+}) and then belongs to $GDG_p^+ $. Similarly, the result (\ref{3.111}) in Lemma \ref{proposition 2} tells us that $u^1$ satisfies (\ref{eDG-}) and so $u^1 \in GDG_p^- $. The result of Theorem \ref{main theorem for section 2} follows from Theorem \ref{theorem 21}.
\end{proof}

As a final remark of this section, we mention that we only considered just now integral functionals in 4-dimensional Euclidean spaces. The same idea can be used to deal with integral functionals
\begin{equation*}\label{cal F}
{\cal F} (u,\Omega) =\int_\Omega f(x,Du(x))dx
\end{equation*}
in $n$-dimensional Euclidean spaces, $n\ge 3$, with the integrand has splitting form
\begin{equation}\label{splitting structure}
f(x,\xi) =\sum_{\alpha =1}^n F_\alpha (x,\xi^\alpha) +\sum _{\beta =1}^{n(n-1)/2} G_\beta (x,({\rm adj}_2 \xi) ^\beta ) +\cdots + H(x,\det \xi).
\end{equation}
If one assume that there exist exponents
$$
1<p_1\le n, \ 1<p_2, p_3,\cdots, p_{n-1}, \ 1\le p_n,
$$
constants $c_1,c_3>0$, $c_2\ge 0$, and nonnegative functions
$$
a_i(x) \in L_{loc} ^{\sigma } (\Omega), \ \sigma >\frac n {p_1}, \ i=1,2,\cdots,n,
$$
such that
\begin{equation}\label{growth conditions}
\begin{array}{llll}
&\displaystyle c_1 |\lambda |^{p_1} -c_2 \le F_\alpha (x,\lambda )\le c_3 (|\lambda |^{p_1}+1) +a_1(x), \ \forall \lambda \in \mathbb R^n,  \ \alpha =1,2,\cdots, n, \\
& \displaystyle c_1 |\eta|^{p_2} -c_2  \le G_\beta (x,\eta) \le c_3 (|\eta|^{p_2}+1) +a_2(x), \ \forall \eta \in \mathbb R^{n(n-1)/2},  \ \beta =1,2,\cdots, n(n-1)/2, \\
 & \displaystyle \qquad \qquad \qquad \qquad  \cdots  \\
& \displaystyle 0 \le H(x,t) \le c_3 (|t| ^{p_n}+1) +a_n(x), \ \forall t \in \mathbb R,
\end{array}
\end{equation}
then the following theorem is immediate.

\begin{theorem}\label{main theorem for section 22}
Let $f$ has splitting structure as in (\ref{splitting structure}), and satisfy the growth conditions (\ref{growth conditions}).
Let $u\in W^{1,1}_{loc} (\Omega,\mathbb R^n)$ be a local minimizer of $\mathcal{F}$ in the sense that $f(x,Du(x)) \in L^1 _{loc} (\Omega)$ and
$$
{\cal F} (u,{\rm supp}\varphi ) \le {\cal F} (u+\varphi,{\rm supp}\varphi).
$$
Let
$$
1\le p_n <p_{n-1} <\cdots <p_2<p_1\le n.
$$
Assume
\begin{equation*}\label{relation for theorem 2.21}
\frac {p_1}{p_1^*} <1- \max \left\{\frac {p_2p_1^*}{p_1(p_1^*-p_2)}, \frac {p_3p_1^*}{p_2(p_1^*-p_3)}, \cdots, \frac 1 \sigma \right\},
\end{equation*}
where $p_1^* =\frac {np_1}{n-p_1}$, if $p_1<n$, and, if $p_1=n$, then $p_1^*$ is any $\nu >n$.

Then all local minimizers $u\in W_{loc} ^{1,1} (\Omega, \mathbb R^{n})$ of $\cal F$ are locally bounded and locally H\"older continuous.
\end{theorem}

As a corollary of Theorem \ref{main theorem for section 22}, we take $n=3$, $p_1=p$, $p_2=q$ and $p_3=r$, we then have
\begin{corollary}
Let $f$ be as in (\ref{integrand from CLM}).
We assume that there exist exponents $1\le r<q<p\le 3$, constants $c_1,c_3>0$, $c_2\ge 0$ and functions
$$
0\le a,b,c \in L^\sigma (\Omega),  \ \sigma >\frac 3 p ,
$$
such that
\begin{equation*}
\begin{array}{llll}
&\displaystyle c_1 |\lambda |^p-c_2 \le F_\alpha (x,\lambda ) \le c_3 (|\lambda |^p +1) +a(x), \  \forall \lambda \in \mathbb R^3, \ \alpha =1,2,3, \\
&\displaystyle c_1 |\lambda |^q-c_2 \le G_\alpha (x,\lambda ) \le c_3 (|\lambda |^q +1) +b(x), \  \forall \lambda \in \mathbb R^3, \ \alpha =1,2,3, \\
&\displaystyle 0\le H(x,t)\le c_3 (|t|^r+1),\  \forall t \in \mathbb R.
\end{array}
\end{equation*}
Assume
$$
\frac {p}{p^*} <\min \left\{1-\frac {qp^*}{ p(p^*-q)}, 1-\frac {rp^*}{q (p^*-r)}, 1-\frac 1 \sigma  \right\},
$$
where $p^*=\frac {3p}{3-p}$, if $p<3$, and, if $p=3$, then $p^*$ is any $\nu >3$.

Then all the local minimizers $u\in W_{loc} ^{1,1} (\Omega, \mathbb R^3)$ of $\cal F$ are locally bounded and locally H\"older continuous.
\end{corollary}

Let us compare the above corollary with Theorem 2.1 in \cite{Cupini-Leonetti-Mascolo}. We find that the assumptions are the same, but the result of Theorem 2.1 in \cite{Cupini-Leonetti-Mascolo} is
locally bounded, while the results of the above lemma are, except for that $u$ is locally bounded, $u$ is locally H\"older continuous.

\section{A degenerate linear elliptic equation.}
In this section we shall give another application of Theorem \ref{theorem 21} to regularity property of weak solutions of a linear elliptic equation with the form
\begin{equation}\label{linear equation}
-{\rm div} (a(x)D u)=-{\rm div} F, \quad {\rm in }\  \Omega,
\end{equation}
here $\Omega$ stands for an open bounded subset of $\mathbb R^n$, $n\geq 2$, and
\begin{equation}\label{F}
F\in(L_{loc}^{2}(\Omega))^{n},
\end{equation}
\begin{equation}\label{a(x)}
0< a(x)\leq \beta<+\infty, \ \ {\rm a.e. } \ \Omega.
\end{equation}

\begin{definition} A function $u\in W_{loc} ^{1,2} (\Omega)$ is called a solution to (\ref{linear equation}) if
\begin{equation}\label{test function}
\int_{\Omega}a(x)D uD \varphi \mathrm{d}x =\int_{\Omega} F D \varphi \mathrm{d}x
\end{equation}
for all $\varphi \in W^{1,2} (\Omega)$ with $\mbox {supp}\varphi \Subset \Omega$.
\end{definition}

\begin{rem} Note that, the differential operator $-\mbox {div} (a(x)Du)$ is not coercive on $W_0^{1,2} (\Omega)$, even if it is well defined between $W_0^{1,2} (\Omega)$ and its dual. To see that it is sufficient to consider the sequence
$$
u_m(x) =|x| ^{\frac {m(1-n)}{2 (m+1)}} -1, \ \ m=1,2,\cdots,
$$
and
$$
a(x) =|x|
$$
defined in $\Omega =B_1(0)$. It satisfies
$$
\int_{B_1(0)} |Du_m|^2 {\rm d}x =\int_{B_1(0)}  \frac 1 {|x| ^{\frac {m(n+1)+2}{m+1}}} {\rm d}x =+\infty, \ \mbox { for } m\ge n-2,
$$
thus
\begin{equation}\label{4.111}
\|u_m\| _{W_0^{1,2} (\Omega)} =+\infty, \  \mbox { for } m\ge n-2,
\end{equation}
and, at the same time
\begin{equation}\label{4.222}
\int_{B_1(0)} a(x) |Du_m |^2 {\rm d}x =\int_{B_1(0)} \frac 1 {|x| ^{\frac {nm+1}{m+1}}} {\rm d}x <+\infty, \ \mbox { for all } m =1,2,\cdots.
\end{equation}
(\ref{4.111}) together with (\ref{4.222}) implies
$$
\frac 1 {\|u_m\| _{W_0^{1,2} (\Omega)} } \int_{B_1(0)} a(x) |Du_m |^2 {\rm d}x =0, \ \mbox { as } m\rightarrow +\infty.
$$
\end{rem}

For some developments related to elliptic partial differential equations, we refer to the classical books \cite{LU} by Lady\v{z}enskaya and Ural'tseva, \cite{GT} by Gilbarg and Trudinger,  \cite{HKM} by Heinonen, Kilpel\"ainen and Martio and \cite{BC} by Boccardo and Croce.  In \cite{Koskela-Manfredi}, the authors considered
\begin{equation*}\label{a harmonic equation}
-\mbox {div} {\cal A} (x,D u(x)) =0
\end{equation*}
with
$$
\alpha (x) |\xi|^p \le {\cal A} (x,\xi) \xi \le \beta(x) |\xi|^p,
$$
and obtained among other things  that weak solutions of (\ref{a harmonic equation}) are weakly monotone in the sense of \cite{Manfredi} if $\beta (x) \in L^\infty (\Omega)$ and $\alpha >0$ a.e. $\Omega$. Related results can be found in \cite{Latvala}.

We remark that, under (\ref{F}) and (\ref{a(x)}), the integrals in (\ref{test function}) are well-defined. We remark also that the main feature of (\ref{linear equation}) lies in the fact that the function $a(x)$ can be sufficiently close to 0, even if it is always positive. If $a(x)$ is uncontrollable near 0, then one can not expect any regularity property of solutions to (\ref{linear equation}), but for $a(x)$ satisfying
\begin{equation}\label{condition for a(x)}
\frac {1}{ a(x)}\in L_{loc}^r(\Omega),\ \  r>n(n+1),
\end{equation}
and $F(x)$ satisfying
\begin{equation}\label{F2}
F \in(L _{loc}^m( \Omega))^{n}, \ m> \frac {n(n+1)}{n-1},
\end{equation}
one can prove that any solution  $u\in W_{loc} ^{1,2} (\Omega)$ of (\ref{linear equation}) is locally bounded and locally H\"older continuous, as the following theorem shows.

\begin{theorem}\label{main theorem 2}
Suppose (\ref{a(x)}), (\ref{condition for a(x)}) and (\ref{F2}), then any weak solution $u\in W_{loc} ^{1,2} (\Omega)$ of (\ref{linear equation}) is locally bounded and locally H\"older continuous.
\end{theorem}

\begin{proof} Let $B_{R_0} (x_0) \Subset \Omega$.  Let $s,t $ be such that $0< s<t <R_0 $. Consider a cut-off function  $ \eta \in C_{0}^{\infty}\left(B_{t}\right)$  satisfying the following assumptions
\begin{equation*}\label{test}
0 \leq \eta \leq 1, \  \eta \equiv 1   \mbox { in }   B_{s}\left(x_{0}\right) \mbox { and } \ |D \eta| \leq \frac{2}{t-s} .
\end{equation*}
Let for $k \in \mathbb{R}$,
$$
A_{k}=\{x\in\Omega:u(x)>k\} \ \mbox { and } \  A_{k,t}= A_{k} \cap B_{t}.
$$
Take $\varphi=\eta^{2}(u-k)_{+}$ as a test function in the weak formulation (\ref{test function}). Note that
$$
D \varphi=\big( \eta^{2}D u +2\eta D\eta (u-k) \big) \cdot 1_{A_{k,t}},
$$
where $1_{E} (x)$ is the characteristic function of the set $E$, that is, $1_{E} (x)=1$ if $x\in E$ and $1_{E} (x)=0$ otherwise. Thus
\begin{equation*}
\begin{array}{llll}
&\displaystyle \int_{ A_{k,t}} \eta^{2} a(x) |D u| ^2 \mathrm{d}x +2\int_{A_{k,t}}  \eta  a(x)D{u}D \eta (u-k)\mathrm{d}x \\
= &\displaystyle \int_{A_{k,t}} \eta^{2}  FD u \mathrm{d}x +2\int_{A_{k,t}}\eta  F D \eta (u-k) \mathrm{d}x.
\end{array}
\end{equation*}
(\ref{a(x)}) allows us to derive
\begin{equation}\label{2.40}
\begin{array}{llll}
&\displaystyle  \int_{ A_{k,t}} a(x) |\eta D  u |^2 \mathrm{d}x \\
 \leq &\displaystyle   2\beta \int_{A_{k,t}} \eta |D {u}| |D \eta| (u-k)\mathrm{d}x + \int_{A_{k,t}}\eta^{2}|F| |D  u| \mathrm{d}x \\
 &\displaystyle  +2\int_{ A_{k,t}}\eta |F||D  \eta| (u-k)\mathrm{d}x.
\end{array}
\end{equation}
Let
$$
\frac {2n}{n+1}< \delta < \frac {2n+1}{n+1} <2
$$
be fixed. Using Young inequality with exponents $\frac{2}{\delta}$ and $\frac{2}{2-\delta}$ and (\ref{2.40}) one has
\begin{equation*}\label{2.36}
\begin{array}{llll}
&\displaystyle \int_{ A_{k,t}} |\eta D  u|^\delta \mathrm{d}x \\
= &\displaystyle \int_{A_{k,t}} |\eta D  u |^\delta  a(x)^{\frac{\delta}{2}} \Big( \frac{1}{a(x)}\Big )^{\frac{\delta}{2}}\mathrm{d}x\\
\leq &\displaystyle  c\Big( \int_{A_{k,t}} a(x) |\eta D  u| ^2\mathrm{d}x +\int_{A_{k,t}}\Big(\frac{1}{ a(x)}\Big)^{\frac{\delta}{2-\delta}}\mathrm{d}x\Big)
\end{array}
\end{equation*}
\begin{equation}\label{2.36}
\begin{array}{llll}
\leq &\displaystyle  c\Big( \int_{A_{k,t}} \eta  |D {u}||D \eta| (u-k)\mathrm{d}x + \int_{A_{k,t}}\eta^{2} |F||D  u| \mathrm{d}x\\
&\displaystyle  +\int_{ A_{k,t}}\eta |F||D  \eta| (u-k)\mathrm{d}x+\int_{A_{k,t}}\Big(\frac{1}{\alpha(x)}\Big)^{\frac{\delta}{2-\delta}}\mathrm{d}x \Big) \\
:= &\displaystyle c(I_1+I_2+I_3+I_4) .
\end{array}
\end{equation}
Our next goal is to estimate each term in the right hand side of (\ref{2.36}). We use Young inequality with exponent $\delta $ and $\delta '$ in order to estimate
\begin{equation}\label{3.2.2}
\begin{split}
I_1 =& \int_{A_{k,t}} \eta |D {u}| |D \eta| (u-k)\mathrm{d}x\\
\leq & \varepsilon\int_{ A_{k,t}} |\eta D  u|^\delta \mathrm{d}x +c(\varepsilon)\int_{ A_{k,t}}(|D \eta| (u-k))^{{\delta}'} \mathrm{d}x \\
\leq&  \varepsilon\int_{ A_{k,t}} |\eta D  u| ^\delta \mathrm{d}x + c(\varepsilon) \int_{A_{k,t}}\Big (\frac{u-k}{t-s}\Big )^{{\delta}'}\mathrm{d}x.
\end{split}
\end{equation}
We notice that
\begin{equation}\label{delta}
\frac {2n}{n+1} <\delta \Longleftrightarrow \delta '<\frac {2n}{n-1}
\end{equation}
which together with fact $m>\frac {n(n+1)}{n-1}$ in (\ref{F2}) implies $\delta ' <m$.

We use Young inequality again with exponents $\delta$ and $\delta '$ and H\"older inequality with exponents $\frac{m}{{\delta}'}$ and $\frac{m}{{m-\delta}'}$ to derive
\begin{equation}\label{3.2.3}
\begin{split}
I_2 = & \int_{A_{k,t}}\eta^{2}|F| |D  u| \mathrm{d}x\\
\leq &\varepsilon\int_{ A_{k,t}} |\eta D  u|^\delta \mathrm{d}x +c(\varepsilon)\int_{ A_{k,t}}(\eta|F| )^{{\delta}'} \mathrm{d}x \\
\leq &\varepsilon\int_{ A_{k,t}} |\eta D  u|^\delta \mathrm{d}x +c(\varepsilon)\int_{ A_{k,t}}|F|^{{\delta}'} \mathrm{d}x \\
\leq & \varepsilon\int_{ A_{k,t}} |\eta D  u| ^\delta \mathrm{d}x +c(\varepsilon) \|F\|^{{\delta}'}_{L^{m}(B_{R_0})}|A_{k,t}|^{1-\frac{{{\delta}'}}{m}}.
\end{split}
\end{equation}

Similarly, we use Young inequality again with exponents $\delta$ and $\delta '$ and H\"older inequality with exponents $\frac{m}{{\delta}}$ and $\frac{m}{{m-\delta}}$ (note that $\delta <2\le n<m$),
\begin{equation}\label{3.2.4}
\begin{split}
I_3= & \int_{ A_{k,t}}\eta |F|| D  \eta| (u-k)\mathrm{d}x\\
\leq & c \int_{ A_{k,t}}(| D \eta| (u-k))^{{\delta}'} \mathrm{d}x +c \int_{A_{k,t}}(|F| \eta)^{\delta} \mathrm{d}x \\
\leq & c \int_{A_{k,t}}\Big (\frac{u-k}{t-s}\Big )^{{\delta}'}\mathrm{d}x +c \int_{A_{k,t}}|F|^{\delta} \mathrm{d}x \\
\leq & c \int_{A_{k,t}}\Big (\frac{u-k}{t-s}\Big )^{{\delta}'}\mathrm{d}x +c\|F\|^{\delta}_{L^{m}(B_{R_0})}|A_{k,t}|^{1-\frac{{\delta}}{m}}.
\end{split}
\end{equation}

The assumption $\delta <\frac {2n+1}{n+1}$ together with $r>n(n+1)$ and $n\ge 2$ implies $r(2-\delta) >\delta$.
Using H\"older inequality with exponents $\frac{r(2-\delta)}{\delta} $ and $\frac{r(2-\delta)}{r(2-\delta)-\delta} $,
 \begin{equation}\label{3.2.5}
\begin{split}
I_4= & \int_{A_{k,t}}\Big(\frac{1}{\alpha(x)}\Big)^{\frac{\delta}{2-\delta}}\mathrm{d}x \\
 \leq & c \Big(\int_{A_{k,t}}\Big( \frac{1}{\alpha(x)} \Big)^{r}\mathrm{d}x\Big)^{\frac{\delta}{(2-\delta)r}}\Big(\int_{A_{k,t}}1\mathrm{d}x\Big)^{1-\frac{\delta}{(2-\delta)r}}\\
 \leq & c \Big\|\frac{1}{\alpha(x)}\Big\|^{\frac{\delta}{2-\delta}}_{L^{r}(B_{R_0})}|A_{k,t}|^{1-\frac{{\delta}}{(2-\delta)r}}.
\end{split}
\end{equation}
Substituting (\ref{3.2.2}), (\ref{3.2.3}), (\ref{3.2.4}) and (\ref{3.2.5}) into (\ref{2.36}), one has
\begin{equation*}
\begin{split}
& \int_{ A_{k,t}} |\eta D  u |^\delta \mathrm{d}x \\
\leq & 2\varepsilon\int_{ A_{k,t}} |\eta D  u| ^\delta \mathrm{d}x + c\int_{A_{k,t}}\Big (\frac{u-k}{t-s}\Big )^{{\delta}'}\mathrm{d}x+c\|F\|^{{\delta}'}_{L^{m}(\Omega)}|A_{k,t}|^{1-\frac{{{\delta}'}}{m}}\\
& +c\|F\|^{\delta}_{L^{m}(\Omega)}|A_{k,t}|^{1-\frac{{\delta}}{m}}+
\Big\|\frac{1}{\alpha(x)}\Big\|^{\frac{\delta}{2-\delta}}_{L^{r}(B_{R_0})} |A_{k,t}|^{1-\frac{{\delta}}{(2-\delta)r}}.
\end{split}
\end{equation*}
Since $\delta <2<\delta '$, then $1-\frac {\delta '}{m} <1-\frac {\delta}{m}$, thus
$$
|A_k|^{1-\frac {\delta}{m}} =|A_k|^{\frac {\delta '-\delta }{m}} |A_k|^{1-\frac {\delta '}{m}} \le |\Omega|^{\frac {\delta '-\delta }{m}} |A_k|^{1-\frac {\delta '}{m}}.
$$
We use this fact, and take $\varepsilon =\frac 1 4$, then
\begin{equation*}
\begin{split}
& \int_{ A_{k,s}} | D  u| ^\delta \mathrm{d}x \leq \int_{ A_{k,t}} |\eta D  u| ^\delta \mathrm{d}x \\
 \leq& c \left( \int_{A_{k,t}}\Big (\frac{u-k}{t-s}\Big )^{{\delta}'}\mathrm{d}x+ |A_{k,t}|^{1-\frac{{{\delta}'}}{m}}
  + |A_{k,t}|^{1-\frac{{\delta}}{m}}+|A_{k,t}|^{1-\frac{{\delta}}{(2-\delta)r}} \right)\\
\leq & c \left( \int_{A_{k,t}}\Big (\frac{u-k}{t-s}\Big )^{{\delta}'} \mathrm{d}x+ |A_{k,t}|^{1-\frac{{{\delta}'}}{m}}
 + |A_{k,t}|^{1-\frac{{\delta}}{(2-\delta)r}} \right).
\end{split}
\end{equation*}
Let $\tilde{\theta}=\min \left\{1-\frac{{{\delta}'}}{m},1-\frac{{\delta}}{(2-\delta)r}\right\}$. (\ref{condition for a(x)}) and (\ref{F2}) together with the condition on $\delta$ in (\ref{delta}) ensure $\tilde \theta >1-\frac \delta n$. The condition $\delta >\frac {2n}{n+1}$ implies
$Q=\delta ' <\delta ^*$. The above inequality has the form
$$
\int_{A_{k,s}} |Du|^\delta  \mathrm{d}x \le c \left[\int_{A_{k,t}} \left( \frac {u-k}{t-s} \right) ^Q +|A_{k,s}| ^{1-\frac \delta n +\varepsilon}\right]
$$
with $s<t$, $\delta < Q< \delta ^*$ and $\varepsilon >0$, thus $u\in GDG _\delta ^+ $.

In order to prove that $u\in GDG _\delta ^- $ it suffices that $\tilde u = -u\in GDG _\delta ^+ $.  We note that $\tilde u $ satisfies
$$
-\mbox {div} (a(x)  D  \tilde u ) =-\mbox {div} \tilde F,
$$
with $\tilde F =-F$ and
$$
\tilde F \in (L_{loc} ^2 (\Omega))^n.
$$
Reasoning as above, one can derive that $-u\in GDG _\delta ^+ $, which together with $u\in GDG _\delta ^+ $ implies $u\in GDG _\delta $.
 Theorem \ref{theorem 21} gives the result.
\end{proof}

\section{A nonlinear elliptic equation.}
This section gives an application of Theorem \ref{theorem 21} to regularity property of weak solutions of nonlinear elliptic equations of the form
\begin{equation}\label{nonlinear equation}
-{\rm div} {\cal A}(x,u(x), D  u(x))=f(x), \quad {\rm in }\  \Omega,
\end{equation}
here $\Omega$ stands for an open bounded subset of $\mathbb R^n$, $n\geq 2$. We assume that there exists positive constants $\alpha ,\beta$, and
$$
1<p\le n,\  \frac {np}{n+1}<\bar p \le p,
$$
such that for all $\xi \in \mathbb R^n$,
\begin{equation}\label{elliptic condition}
{\cal A} (x,s,\xi) \xi \ge \alpha |\xi| ^{\bar p }
\end{equation}
and
\begin{equation}\label{contrallable growth condition}
|{\cal A} (x,s,\xi)|\le \beta |\xi| ^{p-1}.
\end{equation}
As far as the function $f$ in (\ref{nonlinear equation}) is concerned, we assume
\begin{equation}\label{condition for f 22}
f\in L_{loc}^m (\Omega), \ m >\frac n {p-1}.
\end{equation}

\begin{definition}

A function $u\in W_{loc} ^{1,p} (\Omega)$ is said to be a (weak) solution to (\ref{nonlinear equation}) if
\begin{equation}\label{definitioin 5.1}
\int_\Omega {\cal A} (x,u, D  u)  D  \varphi {\rm d}x =\int_\Omega f\varphi {\rm d}x
\end{equation}
for all $\varphi \in W^{1,p}(\Omega)$ with compact support.
\end{definition}

We note that in (\ref{contrallable growth condition}), the growth of $|{\cal A}(x,s,\xi)|$ is controlled by $|\xi|^{p-1}$, while in (\ref{elliptic condition}), ${\cal A}(x,s,\xi)$ is coercive with $|\xi| ^{\bar p }$, where $\bar p $ may be smaller than $p$ but greater than $\frac {np}{n+1}$. If $\bar p =p$ then we are in the natural growth conditions. But if $\bar p <p$ then we are in the sense of general growth conditions or nonstandard conditions. For some results related to integral functionals with nonstandard growth conditions, we refer to Marcellini \cite{Marcellini1,Marcellini2,Marcellini3,Marcellini4,Marcellini5,Marcellini6,Marcellini7}, Bogelein-Duzaar-Marcellini \cite{Bogelein-Duzaar-Marcellini} and Esposito-Leonetti-Mingione \cite{Esposito-Leonetti-Mingione} and the references therein.

We prove the following

\begin{theorem}\label{theorem 5.1}
Assume (\ref{elliptic condition}), (\ref{contrallable growth condition}) and (\ref{condition for f 22}). Then all weak solutions $u\in W_{loc} ^{1,p} (\Omega)$ to (\ref{nonlinear equation}) are locally bounded and locally H\"older continuous.
\end{theorem}

\begin{proof}
For $B_{R_1} \Subset \Omega$, $0<s<t\le R_1 \le 1$, let us take $\eta \in C_0^\infty (B_{t})$ as follows
$$
0\le \eta \le 1, \ \eta \equiv 1 \mbox { in } B_s \mbox { and } | D  \eta | \le \frac 2 {t-s}.
$$
If we take
$$
\varphi =\eta (u-k)_+ \in W_0^{1,p} (B_t),
$$
then
$$
 D  \varphi =\left[ D  \eta (u-k) + \eta  D  u\right] 1_{A_{k,t}},
$$
where $A_{k,t} =\{u>k\}\cap B_t$. We use the above $\varphi$ as a test function in the weak formulation (\ref{definitioin 5.1}) and we have
$$
\int_{A_{k,t}} {\cal A} (x,u, D  u) \left[ D  \eta (u-k) + \eta  D  u\right] {\rm d}x =\int_{A_{k,t}} f \eta (u-k) {\rm d}x .
$$
We use the above inequality, (\ref{elliptic condition}), (\ref{contrallable growth condition}) and we derive
\begin{equation}\label{5.5}
\begin{array}{llll}
&\displaystyle \int_{A_{k,s}} | D  u| ^{\bar p } {\rm d}x \le \int_{A_{k,t}} \eta | D  u| ^{\bar p } {\rm d}x\\
\le &\displaystyle \int_{A_{k,t}} {\cal A} (x,u, D  u) \eta  D  u {\rm d}x \\
= &\displaystyle \int_{A_{k,t}} {\cal A} (x,u, D  u)  D  \eta (u-k) {\rm d}x +\int_{A_{k,t}}f \eta (u-k){\rm d}x\\
:=&\displaystyle  I_1 +I_2.
\end{array}
\end{equation}
Using (\ref{contrallable growth condition}) and Young inequality with exponents $\frac {\bar p }{p-1}$ and $\frac {\bar p }{\bar p -p+1}$, $|I_1|$ can be estimated as
\begin{equation}\label{estimate for I1}
\begin{array}{llll}
|I_1| &\le &\displaystyle \beta \int_{A_{k,t}} | D  u |^{p-1} | D  \eta | (u-k) {\rm d}x \\
&\le &\displaystyle \varepsilon \int_{A_{k,t}} | D  u| ^{\bar p } {\rm d}x + c(\varepsilon) \int_{A_{k,t}} \left(\frac {u-k}{t-s}\right) ^{\frac {\bar p }{ \bar p -p+1}} {\rm d}x .
\end{array}
\end{equation}
$|I_2|$ can be estimated by using Young inequality with exponents $\frac {\bar p }{p-1}$ and $\frac {\bar p }{\bar p -p+1}$ again as
\begin{equation}\label{estimate for I2}
\begin{array}{llll}
|I_2| &\le &\displaystyle \int_{A_{k,t}} |f|\eta (u-k) {\rm d}x \\
&\le &\displaystyle \int_{A_{k,t}} |f| ^{\frac {\bar p }{ p -1}} {\rm d}x + \int_{A_{k,t}} \left(u-k\right) ^{\frac {\bar p }{ \bar p -p+1}} {\rm d}x \\
&\le &\displaystyle \int_{A_{k,t}} |f| ^{\frac {\bar p }{p -1}} {\rm d}x + \int_{A_{k,t}} \left(\frac {u-k}{t-s}\right) ^{\frac {\bar p }{ \bar p -p+1}} {\rm d}x,
\end{array}
\end{equation}
where we have used the fact $s<t<R_1 \le 1$, which implies $1\le \frac 1 {t-s}$.

Substituting (\ref{estimate for I1}) and (\ref{estimate for I2}) into (\ref{5.5}) one has
\begin{equation}\label{5.555}
\begin{array}{llll}
&\displaystyle \int_{A_{k,s}} | D  u| ^{\bar p } {\rm d}x \\
\le &\displaystyle \varepsilon \int_{A_{k,t}} | D  u| ^{\bar p } {\rm d}x + (c(\varepsilon)+1) \int_{A_{k,t}} \left(\frac {u-k}{t-s}\right) ^{\frac {\bar p }{ \bar p -p+1}} {\rm d}x +\int_{A_{k,t}} |f| ^{\frac {\bar p }{p -1}} {\rm d}x.
\end{array}
\end{equation}
Now we want to eliminate the first term in the right hand side including $| D  u|^{\bar p }$.
We use a very useful lemma for real functions which can be found, for example, in \cite{Giusti} Lemma 6.1 or \cite{Giaquinta} Lemma 3.1.

\begin{lemma}\label{lemma 5.1}
Let $f(\tau)$ be a non-negative bounded function defined for $0\le R_0\le \tau \le R_1$. Suppose that for $R_0 \le s <t\le R_1$ we have
$$
f(s) \le A (t-s) ^{-\alpha} +B +\varepsilon  f(t),
$$
where $A,B,\alpha, \varepsilon $ are non-negative constants, and $\varepsilon  <1$. Then there exists a constant $c$, depending only on $\alpha$ and $\varepsilon  $, such that for every $\rho, R$, $R_0\le \rho <R\le R_1$ we have
$$
f(\rho ) \le c [A(R-\rho ) ^{-\alpha } +B].
$$
\end{lemma}

We let $\rho, R$ fixed with $\rho <R\le R_1$. Thus, from (\ref{5.555}) we deduce for every $s,t$ such that $\rho \le s<t\le R$, it results
\begin{equation}\label{5.5552}
\begin{array}{llll}
&\displaystyle \int_{A_{k,s}} | D  u| ^{\bar p } {\rm d}x \\
\le &\displaystyle \varepsilon \int_{A_{k,t}} | D  u| ^{\bar p } {\rm d}x + (c(\varepsilon)+1) \int_{A_{k,R}} \left(\frac {u-k}{t-s}\right) ^{\frac {\bar p }{ \bar p -p+1}} {\rm d}x +\int_{A_{k,R}} |f| ^{\frac {\bar p }{p -1}} {\rm d}x.
\end{array}
\end{equation}
Applying Lemma \ref{lemma 5.1} in (\ref{5.5552}) we conclude that

\begin{equation}\label{5.99}
\int_{A_{k,\rho}} | D  u| ^{\bar p } {\rm d}x \le c \int_{A_{k,R}} \left(\frac {u-k}{R-\rho}\right) ^{\frac {\bar p }{ \bar p -p+1}} {\rm d}x +\int_{A_{k,R}} |f| ^{\frac {\bar p }{p -1}} {\rm d}x.
\end{equation}
Since $\frac {np}{n+1}<\bar p$, then $Q=\frac {\bar p }{ \bar p -p+1} <\bar p ^*$. The condition $m>\frac n {p-1}$ in (\ref{condition for f 22}) and $\bar p \le n$ imply $\frac {m(p-1)}{\bar p}>1$. H\"older inequality with exponents $\frac {m(p-1)}{\bar p}$ and $\frac {m(p-1)}{m(p-1)-\bar p}$ yields
$$
\int_{A_{k,R}} |f| ^{\frac {\bar p }{p -1}} {\rm d}x \le \left(\int_\Omega |f| ^m {\rm d}x \right) ^{\frac {\bar p }{m (p -1)}} |A_{k,R}| ^{1-\frac {\bar p }{m (p -1)}} .
$$
We use the condition (\ref{condition for f 22}) for $f$ again and we derive $1-\frac {\bar p }{m (p -1)} >1-\frac {\bar p }{n}$. Therefore (\ref{5.99}) has the form
\begin{equation}\label{5.9}
\int_{A_{k,\rho}} | D  u| ^{\bar p } {\rm d}x \le c \int_{A_{k,t}} \left(\frac {u-k}{R-\rho}\right) ^{Q} {\rm d}x +c_*  |A_{k,R}| ^{1-\frac {\bar p }{n} +\varepsilon},
\end{equation}
with $Q=\frac {\bar p }{ \bar p -p+1} <\bar p^*$ and $\varepsilon = \frac {\bar p }{n} -\frac {\bar p }{m(p -1)} >0$. (\ref{5.9}) tells us that $u\in GDG ^+_{\bar p }$. Similarly, one can derive that $u\in GDG ^-_{\bar p }$. The result of Theorem \ref{theorem 5.1} follows from Theorem \ref{theorem 21}.
\end{proof}

\section{A quasilinear elliptic system.}
This section deals with quasilinear elliptic systems in divergence form
\begin{equation}\label{6.1}
\begin{cases}
-{\rm div} (a(x, u(x))Du(x))=-{\rm div} F, & {\rm in }\  \Omega, \\
u(x)=0,  & {\rm on }\  \partial \Omega.
\end{cases}
\end{equation}
where $\Omega $ is a bounded open subset of $\mathbb R^n$, $n\ge 3$, $u: \Omega \subset \mathbb{R}^{n} \rightarrow \mathbb{R}^{N}$, $F:\Omega \rightarrow \mathbb R^{Nn} $ and $a: \Omega \times \mathbb{R}^{N} \rightarrow \mathbb{R}^{N^{2} n^{2}} $ is matrix valued with components  $a_{i, j}^{\alpha, \beta}(x, y) $ where $ \alpha, \beta \in\{1,  \cdots , N\}$ and $ i, j \in\{1, \cdots , n\} $.
We stress that the first line in (\ref{6.1}) consists of s system of $N$ equations of the form
\begin{equation}\label{quasilinear elliptic systems}
-\sum_{i=1}^{n} \frac{\partial}{\partial x_{i}}\left(\sum_{\beta=1}^{N} \sum_{j=1}^{n} a_{i, j}^{\alpha, \beta}(x, u) \frac{\partial}{\partial x_{j}} u^{\beta}\right)= -\sum_{i=1}^{n} \frac{\partial}{\partial x_{i}} F^{\alpha}_{i}, \ \ \   \alpha=1, \cdots, N.
\end{equation}

When $N=1$, that is in the case of one single equation, the celebrated De Giorgi-Nash-Moser theorem ensures that weak solutions $u \in W_0^{1,2}(\Omega)$ are locally bounded and locally H\"older continuous, see section 2.1 in \cite{Mingione}. But in the vectorial case $N \geq 2$, the aforementioned result is no longer true due to the De Giorgi counterexample, see section 3 in \cite{Mingione}, see also \cite{DE,Mooney-Savin,PodioGuidugli,Leonardi}.
So it arises the question of finding additional structural restrictions on the coefficients $ a_{i, j}^{\alpha, \beta} $ that keep away De Giorgi counterexample and allow for local boundedness and local H\"older continuity of weak solutions $u$ to (\ref{6.1}).

In the present work we assume, except for ellipticity of all the coefficients, a condition on the support of off-diagonal coefficients: there exists $ L_{0} \in   [0,+\infty) $ such that for all $L \geq L_{0} $, when $ \alpha \neq \beta $, the following (${\cal A}_3$) holds.
Under such a restriction we are able to prove local boundedness and local H\"older continuity of weak solutions. It is worth to stress out that systems with special structure have been studied in \cite{Yan,Meier} and off-diagonal coefficients with a particular support have been successfully used when proving maximum principles in \cite{Leonardi-Leonetti-Pignotti}, $ L^{\infty} $ -regularity in \cite{Leonardi-Leonetti-Pignotti2}, when obtaining existence for measure data problems in \cite{Leonetti-Rocha-Staicu,Leonetti-Rocha-Staicu2}, and for the degenerate case, in \cite{Geronimo-Leonetti-Macri}. For some other related developments, we refer to \cite{GDHR,GHDR}.

We now list our structural assumptions on the coefficients $a^{\alpha,\beta} _{i,j} (x,y): \Omega \times \mathbb{R}^{N} \rightarrow \mathbb{R} $: for all $ \alpha, \beta \in\{1, \cdots, N\} $ and all $ i, j \in\{1, \cdots, n\} $, we require that $ a_{i, j}^{\alpha, \beta} (x,y)$ satisfies the following conditions:

$\left(\mathcal{A}_{0}\right)$  $a_{i, j}^{\alpha, \beta}(x, y)$ is a Carath\'eodory function, that is, $ x \mapsto a_{i, j}^{\alpha, \beta}(x, y) $ is measurable and $ y \mapsto a_{i, j}^{\alpha, \beta}(x, y) $ is continuous;

$\left(\mathcal{A}_{1}\right)$  (boundedness of all the coefficients) there exists a positive constant $ c>0 $ such that
$$
\left|a_{i, j}^{\alpha, \beta}(x, y)\right| \leq c
$$
for almost all $ x \in \Omega $ and for all $ y \in \mathbb{R}^{N}$;

$\left(\mathcal{A}_{2}\right)$ (ellipticity of all the coefficients) there exists a positive constant $\nu>0 $ such that
$$
\sum_{\alpha, \beta=1}^{N} \sum_{i, j=1}^{n} a_{i, j}^{\alpha, \beta}(x, y) \xi_{i}^{\alpha} \xi_{j}^{\beta} \geq \nu|\xi|^{2}
$$
for almost all $ x \in \Omega $, for all $ y \in \mathbb{R}^{N} $ and for all $ \xi \in \mathbb{R}^{N \times n} $;

$\left(\mathcal{A}_{3}\right)$ (support of off-diagonal coefficients) there exists $ L_{0} \in(0,+\infty) $ such that $ \forall L \geq L_{0} $, when $ \alpha \neq \beta $,
$$
a_{i, j}^{\alpha, \beta}(x, y) \neq 0, y^{\alpha}>L \Rightarrow y^{\beta}>L, \eqno(\mathcal{A}_{3}')
$$
and
$$
a_{i, j}^{\alpha, \beta}(x, y) \neq 0, y^{\alpha}<-L \Rightarrow y^{\beta}<-L. \eqno(\mathcal{A}_{3}'')
$$

We say that a function $ u: \Omega \rightarrow \mathbb{R}^{N} $ is a weak solution of the system (\ref{quasilinear elliptic systems}) if $ u \in W_0 ^{1,2}\left(\Omega, \mathbb{R}^{N}\right) $ and
\begin{equation}\label{6.3}
\int_{\Omega} \sum_{\alpha, \beta=1}^{N} \sum_{i, j=1}^{n} a_{i, j}^{\alpha, \beta}(x, u(x)) D_{j} u^{\beta}(x) D_{i} \varphi^{\alpha}(x) {\rm d} x= \int_{\Omega}\sum_{\alpha=1}^{N}
\sum_{i=1}^{n}F^{\alpha}_{i}D_{i} \varphi^{\alpha}(x) {\rm d}x
\end{equation}
for all $ \varphi \in W_{0}^{1,2}\left(\Omega, \mathbb{R}^{N}\right) $.

\vspace{2mm}

We prove the following

\begin{theorem}\label{theorem 6.1}
 Let $ u \in W_0 ^{1,2}\left(\Omega, \mathbb{R}^{N}\right) $ be a weak solution of system (\ref{quasilinear elliptic systems}) under the assumptions $(\mathcal{A})_0 -(\mathcal{A})_3$.
Suppose $F \in L_{loc}^{m}(\Omega, \mathbb R^{Nn})$, $m>n$. Then $ u $ are locally bounded and locally H\"older continuous.
\end{theorem}

\begin{proof}
 Let $ u \in W_0^{1,2}\left(\Omega, \mathbb{R}^{N}\right) $ be a weak solution of system (\ref{quasilinear elliptic systems}). Let $ \eta: \mathbb{R}^{n} \rightarrow \mathbb{R} $ be the standard cut-off function such that $ 0 \leq \eta \leq 1$, $\eta \in C_{0}^{1}\left(B\left(x_{0}, t\right)\right) $, with $ B\left(x_{0}, t\right) \subset \Omega $ and $ \eta=1 $ in $ B\left(x_{0}, s\right) $, $0<s<t$. Moreover, $ |D \eta| \leq 2 /(t-s) $ in $ \mathbb{R}^{n} $. For every level $ L \geq L_{0} $, let us consider the test function $ \varphi: \mathbb{R}^{n} \rightarrow \mathbb{R}^{N} $ with $ \varphi=\left(\varphi^{1}, \cdots, \varphi^{N}\right) $, where
$$
\varphi^{\alpha}(x):=\eta^{2}(x) \left(u^{\alpha}(x)-L\right)_{+}, \ \ \text { for } \alpha \in\{1, \cdots, N\},
$$
then for all $\alpha \in\{1, \cdots, N\}$  and $i \in\{1, \cdots, n\}$,
$$
D_{i} \varphi^{\alpha}= \left(\eta^{2}  D_{i} u^{\alpha}+2 \eta\left(D_{i} \eta\right) \left(u^{\alpha}-L\right) \right) 1_{\left\{u^{\alpha}>L\right\}},
$$
where $1_E (x)$ is the characteristic function of the set $E$, that is, $1_E (x) =1$ for $x\in E$ and $1_E (x)=0$ otherwise.
Using this test function in the weak formulation (\ref{6.3}) of system (\ref{quasilinear elliptic systems}), we have
\begin{equation*}
\begin{array}{llll}
& \displaystyle \int_\Omega \sum_{\alpha,\beta=1}^{N} \sum_{i,j=1}^{n} a^{\alpha,\beta}_{i,j}(x,u)D_{j}u^{\beta} \eta^{2} D_{i}u^{\alpha} 1 _{\left\{u^{\alpha}>L\right\}}  {\rm d}x  \\
 &+ \displaystyle\int_\Omega \sum_{\alpha,\beta=1}^{N} \sum_{i,j=1}^{n} a^{\alpha,\beta}_{i,j}(x,u)D_{j}u^{\beta} 2 \eta\left(D_{i} \eta\right)  1 _{\left\{u^{\alpha}>L\right\}}\left(u^{\alpha}-L\right)  {\rm d}x \\
= & \displaystyle\int_\Omega \sum_{\alpha=1}^{N} \sum_{i=1}^{n} F_{i}^{\alpha} \eta^{2} D_{i}u^{\alpha} 1 _{\left\{u^{\alpha}>L\right\}}  {\rm d}x \\
 & +  \displaystyle\int_\Omega \sum_{\alpha=1}^{N} \sum_{i=1}^{n} F_{i}^{\alpha} 2 \eta\left(D_{i} \eta\right)  1 _{\left\{u^{\alpha}>L\right\}}\left(u^{\alpha}-L\right)  {\rm d}x .
\end{array}
\end{equation*}
Now, the assumption $\left(\mathcal{A}_{3}^{\prime}\right)$ guarantees that
\begin{equation}\label{6.4}
a^{\alpha,\beta}_{i,j}(x,u(x))  1 _{\left\{u^{\alpha}>L\right\}}(x)=a^{\alpha,\beta}_{i,j}(x,u(x))  1 _{\left\{u^{\alpha}>L\right\}} (x) 1 _{\left\{u^{\beta}>L\right\}}(x),
\end{equation}
when $\alpha \neq \beta$ and $L\geq L_{0}$. It is worthwhile to note that (\ref{6.4}) holds true when $\alpha=\beta$ as well; then
\begin{equation}\label{6.5}
\begin{array}{llll}
& \displaystyle \int_\Omega \sum_{\alpha,\beta=1}^{N} \sum_{i,j=1}^{n} a^{\alpha,\beta}_{i,j}(x,u)D_{j}u^{\beta} 1 _{\left\{u^{\beta}>L\right\}} \eta^{2} D_{i}u^{\alpha} 1 _{\left\{u^{\alpha}>L\right\}}  {\rm d}x  \\
\leq & \displaystyle\int_\Omega \sum_{\alpha,\beta=1}^{N} \sum_{i,j=1}^{n} a^{\alpha,\beta}_{i,j}(x,u)D_{j}u^{\beta}  1 _{\left\{u^{\beta}>L\right\}} 2 \eta\left(D_{i} \eta\right) \left(u^{\alpha}-L\right) 1 _{\left\{u^{\alpha}>L\right\}}  {\rm d}x \\
 & + \displaystyle \int_\Omega \sum_{\alpha=1}^{N} \sum_{i=1}^{n} F_{i}^{\alpha} \eta^{2} (D_{i}u^{\alpha}) 1 _{\left\{u^{\alpha}>L\right\}}  {\rm d}x \\
 & +  \displaystyle  \int_\Omega \sum_{\alpha=1}^{N} \sum_{i=1}^{n} F_{i}^{\alpha} 2 \eta\left(D_{i} \eta\right)  1 _{\left\{u^{\alpha}>L\right\}}\left(u^{\alpha}-L\right)  {\rm d}x \\
:=& I_1+I_2+I_3.
\end{array}
\end{equation}
Now we use the ellipticity assumption $\left(\mathcal{A}_{2}\right)$ with $\xi_{i}^{\alpha}=D_{i}u^{\alpha} 1 _{\left\{u^{\alpha}>L\right\}}$, $\xi_{j}^{\beta}=D_{j}u^{\beta} 1 _{\left\{u^{\beta}>L\right\}}$ and we get
\begin{equation}\label{6.55}
\begin{array}{llll}
& \displaystyle
\nu \int_{\Omega} \eta^{2} \sum_{\alpha=1}^{N} \left|D u^{\alpha}\right|^{2} 1 _{\left\{u^{\alpha}>L\right\}} {\rm d} x \\
\leq & \displaystyle \int_{\Omega} \sum_{\alpha, \beta=1}^{N} \sum_{i, j=1}^{n} a_{i, j}^{\alpha, \beta}  D_{j} u^{\beta} 1 _{\left\{u^{\beta}>L\right\}} \eta^{2}  D_{i} u^{\alpha} 1 _{\left\{u^{\alpha}>L\right\}}  {\rm d}x .
\end{array}
\end{equation}

Our next task is to estimate each term in the right hand side of (\ref{6.5}).

We use the inequality
\begin{equation}\label{basic inequality}
\left(\sum_{\alpha=1}^{N}a_{\alpha}\right)^{2} \leq 2 ^ {N-1} \sum_{\alpha=1}^{N}a_\alpha^2,
\end{equation}
the boundedness of all the coefficients assumption $\left(\mathcal{A}_{1}\right)$ and Young inequality in order to estimate
\begin{equation}\label{6.6}
\begin{array}{llll}
I_1 &\leq & \displaystyle\int_\Omega \sum_{\beta=1}^{N} \sum_{j=1}^{n} (D_{j}u^{\beta})  1 _{\left\{u^{\beta}>L\right\}}\eta \sum_{\alpha=1}^{N} \sum_{i=1}^{n} 2 \left(D_{i} \eta\right) \left(u^{\alpha}-L\right) 1 _{\left\{u^{\alpha}>L\right\}}  {\rm d}x \\
&\leq & \displaystyle \int_\Omega n\sum_{\beta=1}^{N} |Du^{\beta}|  1 _{\left\{u^{\beta}>L\right\}}\eta \sum_{\alpha=1}^{N} 2n |D\eta| \left(u^{\alpha}-L\right) 1 _{\left\{u^{\alpha}>L\right\}}  {\rm d}x \\
&\leq & \varepsilon \displaystyle\int_\Omega n^2 \eta^2\left(\sum_{\beta=1}^{N} |Du^{\beta}|  1 _{\left\{u^{\beta}>L\right\}}\right)^2  {\rm d}x\\
&& \displaystyle \ \  +c(\varepsilon)\displaystyle\int_\Omega n^2 |D\eta|^2 \left(\sum_{\alpha=1}^{N} \left(u^{\alpha}-L\right) 1 _{\left\{u^{\alpha}>L\right\}}\right)^2  {\rm d}x \\
&\leq & \varepsilon \displaystyle\int_\Omega n^2 2 ^ {N-1}\eta^2 \sum_{\beta=1}^{N} |D u^{\beta}|^2  1 _{\left\{u^{\beta}>L\right\}}  {\rm d}x \\
&& \displaystyle \ \   +c(\varepsilon)\displaystyle\int_\Omega n^2 2 ^ {N-1}|D\eta|^2 \sum_{\alpha=1}^{N} \left(u^{\alpha}-L\right)^2  1 _{\left\{u^{\alpha}>L\right\}} {\rm d}x\\
&\leq & \varepsilon \displaystyle\int_\Omega n^2 2 ^ {N-1}\eta^2 \sum_{\beta=1}^{N} |D u^{\beta}|^2  1 _{\left\{u^{\beta}>L\right\}}  {\rm d}x \\
&& \displaystyle \ \   + c(\varepsilon) n^2 2 ^ {N-1} \sum_{\alpha =1}^N \int_{A_{L,t}^{\alpha}}\left(\frac{u^\alpha-L}{t-s}\right)^2  {\rm d}x \\
&\leq & \varepsilon \displaystyle\int_\Omega n^2 2 ^ {N-1}\eta^2 \sum_{\beta=1}^{N} |D u^{\beta}|^2  1 _{\left\{u^{\beta}>L\right\}}  {\rm d}x \\
&& \displaystyle \ \   + c(\varepsilon) n^2 2 ^ {N-1} \sum_{\alpha=1}^{N}\left[\int_{A_{L,t}^{\alpha}}\left(\frac{u^\alpha-L}{t-s}\right)^Q  {\rm d}x +|A_{L,t}^{\alpha}|\right],
\end{array}
\end{equation}
where we have used the notation
$$
A_{L,t}^{\alpha}=\left\{x\in\Omega:u^{\alpha}>L\right\} \cap B(x_0,t).
$$

We use (\ref{basic inequality}) again, Young inequality and H\"older inequality to derive
\begin{equation}\label{6.9}
\begin{array}{llll}
I_2 &\displaystyle =\int_\Omega \sum_{\alpha=1}^{N} \sum_{i=1}^{n} F_{i}^{\alpha} \eta^{2} (D_{i}u^{\alpha}) 1 _{\left\{u^{\alpha}>L\right\}}  {\rm d}x\\
&\leq \displaystyle \int_\Omega n \sum_{\alpha=1}^{N}  |F^{\alpha}| \eta^{2} |Du^{\alpha}|  1 _{\left\{u^{\alpha}>L\right\}}  {\rm d}x \\
&\leq  \varepsilon \displaystyle\int_\Omega n\left( \sum_{\alpha=1}^{N}  |Du^{\alpha}|  1 _{\left\{u^{\alpha}>L\right\}}\eta\right)^{2} {\rm d}x +c(\varepsilon)\displaystyle\int_\Omega n \left( \sum_{\alpha=1}^{N}  |F^\alpha| 1 _{\left\{u^{\alpha}>L\right\}}\eta\right)^{2} {\rm d}x \\
& \leq \varepsilon \displaystyle\int_\Omega n 2 ^{N-1} \eta^2 \sum_{\alpha=1}^{N}  |Du^{\alpha}|^2  1 _{\left\{u^{\alpha}>L\right\}} {\rm d}x +\displaystyle c(\varepsilon )n 2 ^{N-1} \sum_{\alpha=1}^{N} \int_{A_{L,t}^{\alpha}} |F^\alpha|^2  {\rm d}x  \\
& \leq \varepsilon \displaystyle\int_\Omega n2 ^{N-1} \eta^2 \sum_{\alpha=1}^{N}  |Du^{\alpha}|^2  1 _{\left\{u^{\alpha}>L\right\}} {\rm d}x +\displaystyle c \sum_{\alpha=1}^{N}\|F^\alpha\|^{2}_{L^{m}(B_t)}|A_{L,t}^{\alpha}|^{1-\frac{2}{m}}.
\end{array}
\end{equation}
Similarly, for any $2<Q<2^*$,
\begin{equation}\label{6.10}
\begin{array}{llll}
I_3 & \displaystyle = \int_\Omega \sum_{\alpha=1}^{N} \sum_{i=1}^{n} F_{i}^{\alpha} 2 \eta\left(D_{i} \eta\right)  1 _{\left\{u^{\alpha}>L\right\}}\left(u^{\alpha}-L\right)  {\rm d}x \\
& \leq \displaystyle 2 \int_\Omega n\sum_{\alpha=1}^{N} |F^{\alpha}| \eta\left|D\eta\right|  1 _{\left\{u^{\alpha}>L\right\}}\left(u^{\alpha}-L\right) {\rm d}x \\
& \le 2n\displaystyle \sum_{\alpha=1}^{N} \int_{A_{L,t}^{\alpha}}|F^{\alpha}| \left|D\eta\right|\left(u^{\alpha}-L\right)  {\rm d}x \\
& \leq 2n \displaystyle \sum_{\alpha=1}^{N} \left[\int_{A_{L,t}^{\alpha}}\left|D\eta\right|^Q\left(u^{\alpha}-L\right)^Q  {\rm d}x +\int_{A_{L,t}^{\alpha}}|F^{\alpha}|^{Q'}  {\rm d}x  \right]\\
& \leq 2n \displaystyle \sum_{\alpha=1}^{N} \int_{A_{L,t}^{\alpha}}\left(\frac{u^{\alpha}-L}{t-s}\right)^Q  {\rm d}x  +\displaystyle 2n  \sum_{\alpha=1}^{N}\|F^\alpha\|^{Q'}_{L^{m} (B_t)}|A_{L,t}^{\alpha}|^{1-\frac{Q'}{m}}.
\end{array}
\end{equation}
Substituting the estimates (\ref{6.55}), (\ref{6.6}), (\ref{6.9}) and (\ref{6.10}) into (\ref{6.5}), we arrive at
\begin{equation*}
\begin{array}{llll}
&\displaystyle \sum_{\alpha=1}^{N} \int_{A_{L,s}^{\alpha}} \left|D u^{\alpha}\right|^{2} {\rm d}x \le \int_{\Omega} \eta^{2} \sum_{\alpha=1}^{N} \left|D u^{\alpha}\right|^{2} 1 _{\left\{u^{\alpha}>L\right\}}  {\rm d}x  \\
\leq & \displaystyle c \sum_{\alpha=1}^{N} \int_{A_{L,t}^{\alpha}}\left(\frac{u^{\alpha}-L}{t-s}\right)^Q  {\rm d}x  + c \sum_{\alpha=1}^{N}\left[|A_{L,t}^{\alpha}|+|A_{L,t}^{\alpha}|^{1-\frac{2}{m}}+|A_{L,t}^{\alpha}|^{1-\frac{Q'}{m}}\right].
\end{array}
\end{equation*}
Since $2<Q<2^*$, then $1-\frac{2}{m}<1-\frac{Q'}{m}<1$, thus
$$
|A_{L,t}^{\alpha}|^{1-\frac{Q'}{m}}=|A_{L,t}^{\alpha}|^{\frac{2-Q'}{m}}|A_{L,t}^{\alpha}|^{1-\frac{2}{m}}\leq |\Omega|^{\frac{2-Q'}{m}}|A_{L,t}^{\alpha}|^{1-\frac{2}{m}},
$$
$$
|A_{L,t}^{\alpha}|=|A_{L,t}^{\alpha}|^{\frac{2}{m}}|A_{L,t}^{\alpha}|^{1-\frac{2}{m}}\leq |\Omega|^{\frac{2}{m}}|A_{L,t}^{\alpha}|^{1-\frac{2}{m}},
$$
we use these facts and we have
\begin{equation*}
\sum_{\alpha=1}^{N}\int_{A_{L,s}^{\alpha}}\left|D u^{\alpha}\right|^{2}  {\rm d}x  \leq c \sum_{\alpha=1}^{N} \int_{A_{L,t}^{\alpha}}\left(\frac{u^{\alpha}-L}{t-s}\right)^Q  {\rm d}x  + c\sum_{\alpha=1}^{N}|A_{L,t}^{\alpha}|^{1-\frac{2}{m}},
\end{equation*}
since condition $m>n$, implies $1-\frac{2}{m}>1-\frac{2}{n}$, we get
\begin{equation*}
\sum_{\alpha=1}^{N}\int_{A_{L,s}^{\alpha}}\left|D u^{\alpha}\right|^{2}  {\rm d}x  \leq c \sum_{\alpha=1}^{N} \int_{A_{L,t}^{\alpha}}\left(\frac{u^{\alpha}-L}{t-s}\right)^Q  {\rm d}x  + c\sum_{\alpha=1}^{N}|A_{L,t}^{\alpha}|^{1-\frac{2}{n}+\varepsilon},
\end{equation*}
with $s<t$, $2<Q<2^*$ and $\varepsilon>0$, thus $u\in GDG^{+}_{2} (N,\Omega,2,Q,c,c,\varepsilon, L_0)$.

Now let $\tilde{u}=-u$. Then $ \tilde{u} \in W_0^{1,2}\left(\Omega ; \mathbb{R}^{N}\right) $. Since $ u $ satisfies  (\ref{6.3}), then $ \tilde{u} $ satisfies
$$
\int_{\Omega} \sum_{\alpha, \beta=1}^{N} \sum_{i, j=1}^{n} \tilde{a}_{i, j}^{\alpha, \beta}(x, \tilde{u}(x)) D_{j} \tilde{u}^{\beta}(x) D_{i} \varphi^{\alpha}(x) d x=\int_{\Omega}\sum_{\alpha=1}^{N} \sum_{i=1}^{n}\tilde{F^{\alpha}_{i}}D_{i} \varphi^{\alpha}(x) d x
$$
for every $ \varphi \in W_{0}^{1,2}\left(\Omega, \mathbb{R}^{N}\right)$, where
\begin{equation}\label{6.11}
\tilde{a}_{i, j}^{\alpha, \beta}(x, y):=a_{i, j}^{\alpha, \beta}(x,-y), \ \ \tilde{F^{\alpha}_{i}}=-F^{\alpha}_{i}.
\end{equation}
We observe that the new coefficients, defined by (\ref{6.11}), readily satisfy the conditions $ \left(\mathcal{A}_{0}\right)$, $\left(\mathcal{A}_{1}\right) $ and $\left(\mathcal{A}_{2}\right) $. Moreover, if $ \alpha \neq \beta $ the coefficients $ \tilde{a}_{i, j}^{\alpha, \beta}(x, y) $ satisfy $ \left(\mathcal{A}_{3}^{\prime}\right) $ provided $ a_{i, j}^{\alpha, \beta}(x, y) $ satisfy $ \left(\mathcal{A}_{3}^{\prime \prime}\right) $. Therefore, one can argue as above on $\tilde{u}$ obtaining $u\in GDG^{-}_{2}(N,\Omega,2,Q,c,c,\varepsilon, L_0)$. Theorem \ref{theorem 6.1} follows from Theorem \ref{theorem 212}.
\end{proof}

\vspace{3mm}

\noindent {\bf Acknowledgments:} The first author thanks NSFC(12071021), NSF of Hebei Province (A2019201120) and the Key Science and Technology Project of
Higher School of Hebei Province (ZD2021307) for the support.

\rm 

\footnotesize \baselineskip 9pt
\begin{thebibliography}{17}
\bibitem{BC} L.Boccardo, G.Croce, Elliptic partial differential equations, De Gruyter Studies in Mathematics, vol 55, De Gruyter, 2014.

\bibitem{Bogelein-Duzaar-Marcellini} V.B\"ogelein, F.Duzaar, P.Marcellini, Parabolic equations with $(p,q)$-growth, J. Math. Pures Appl., 2013, 100, 535-563.

\bibitem{Carozza-Gao-Giova-Leonetti} M.Carozza M., H.Gao, R.Giova, F.Leonetti,  A boundedness result for minimizers of some polyconvex integrals. J. Optim. Theory Appl., 2018, 178, 699-725.

\bibitem{Castro} A.Di Castro, T.Kuusi, G.Palaticci, Local behavior of fractional $p$-minimizers, Ann. Inst. H. Poincar\'e Analyse non-lineaire, 2016, 33, 1279-1299.

\bibitem{Cozzi} M.Cozzi, Regularity results and Harnack inequalities for minimizers and solutions of nonlocal problems: A unified approach via fractional De Giorgi classes, J. Funt. Anal., 2017, 272, 4762-4837.

\bibitem{Cupini-Leonetti-Mascolo} G.Cupini, F.Leonetti, E.Mascolo, Local boundeness for minimizers of some policonvex integrals, Arch. Ration. Mech. Anal., 2017, 224, 269-289.

\bibitem{CFLM} G.Cupini, M.Focardi, F.Leonetti, E.Moscolo, On the H\"older continuity for a class of variational problems, Adv. Nonlinear Anal., 2020, 9, 1008-1025.

\bibitem{BT} Di Benedetto E., N.Trudinger, Harnack inequalities for quasi-minima of variatioinal integrals, Ann. Inst. H. Poincar\'e Analyse non-lineaire, 1984, 1, 295-308.

\bibitem{DE} E. De Giorgi, Un esempio di estremali discontinue per un problema variazionale di tipo ellittico, Boll. Un. Mat. Ital. 1968, 4, 135-137.

\bibitem{DiBenedetto} E.DiBenedetto, Partial differential equation, Birkhauser Boston, 1995.

\bibitem{Esposito-Leonetti-Mingione} L.Esposito, F.Leonetti, G.Mingione, Sharp regularity for functionals with $(p,q)$ growth, J. Differ. Equ., 2004, 204, 5-55.

\bibitem{Esposito-Mingione} L.Esposito, G.Mingione, Partial regularity for minimizers of degenerate polyconvex energies, J. Convex Anal., 2001, 8, 1-38.

\bibitem{Frehse} J.Frehse, A note on the H\"older continuity of solutions of variational problems, Abh. Math. Sem. Hamburg, 1975, 43, 59-63.

\bibitem{De-Giorgi} De Giorgi, Sulla differenziabilit\'a and l'analiticit\'a delle estremali degli integrali multipli regolari, Mem. Accad. Sci. Torino (Classe di Sci. mat., fis. and nat.), 1957, 3, 25-43.

\bibitem {Fusco-Hutchinson} N.Fusco, J.Hutchinson, Partial regularity and everywhere continuity for a model problem from non-linear elasticity, J. Austral Math. Soc. Ser. A., 1994, 57, 158-169.

\bibitem{FH} N.Fusco, J.Hutchinson, A direct proof for lower semicontinuity of polyconvex functionals, Manuscripta Math., 1995, 87, 35-50.

\bibitem{Fuchs-Reuling} M.Fuchs, G.Reuling, Partial regularity for certain classes of polyconvex functionals related to non linear elasticity, Manuscripta Math., 1995, 87, 13-26.

\bibitem {Fuchs-Seregin} M.Fuchs, G.Seregin,  H\"older continuity for week extremals of some two-dimensional variational problems related to nonlinear elasticity. Adv. Math. Sci. Appl., 1997, 7, 413-425.

\bibitem {GDHR} H.Gao, H.Deng, M.Huang, W.Ren, Generalizations of Stampacchia lemma and applications to quasilinear elliptic systems, Nonlinear Anal., 2021, 208, 112297.

\bibitem{GHDR} H.Gao, M.Huang, H.Deng, W.Ren, Global integrability for solutions to quasilinear elliptic systems, Manuscripta Math., 2021, 164, 23-37.

\bibitem {Gao-Huang-Ren} H.Gao, M.Huang, W.Ren, Global regularity for minimizers of some anisotropic variational integrals, J. Optim. Theory Appl., 2021, 188: 523-546.

\bibitem{GSR} H.Gao, Y.Shan, W.Ren, Regularity for minimizing sequences of some variational integrals, Sci. China Math.,  https://doi.org/10.1007/s11425-021-1975-5, to appear.

\bibitem{Geronimo-Leonetti-Macri} P. Di Geronimo, F.Leonetti, M. Macr\'i, P.V.Petricca, Existence of bounded solutions for some quasilinear degenerate elliptic systems, Minimax Theory Appl., 2021, 6, 321-340.

\bibitem{Giaquinta} M.Giaquinta, Multiple integrals in the calculus of variations and nonlinear elliptic systems, Princeton University Press, Princeton, NJ, 1983.

\bibitem{GG} M.Giaquinta, E.Giusti, On the regularity of the minima of variational integrals, Acta Math., 1982, 148, 31-46.

\bibitem{Giusti} E.Giusti, Direct methods in the calculus of variations, World Scientific, 2003.

\bibitem {GT} D.Gilbarg, N.S.Trudinger, Elliptic partial differential equations of second order, Springer, 1998.

\bibitem {Hamburger} C.Hamburger, Partial regularity of minimizers of polyconvex variational integrals, Calc. Var. Partial Differential Equations, 2003, 18: 221-241

\bibitem{HKM} J.Heinonen, T.Kilpel\"ainen, O.Martio,  Nonlinear potential theory of degenerate elliptic  equations,  Oxford Univ. Press, Oxford,  1993.

\bibitem{Koskela-Manfredi} P.Koskela, J.J.Manfredi, E.Villamor,  Regularity theory and traces of $\cal A$-harmonic functions, Trans. Amer. Math. Soc. 1996, 348, 755-766.

\bibitem{Kristensen-Taheri} J.Kristensen, A.Taheri, Partial regularity for strong local minimizers in the multi-dimensional calculus of variations, Arch. Rational Mech. Anal., 2003, 170, 63-89.

\bibitem{LU} O.Lady\v{z}enskaya and Ural'ceva, Linear and quasilinear elliptic equations, Academic Press, New York, 1968.

\bibitem{Latvala} V.Latvala, Continuity of weak solutions of elliptic partial differential equations, Ark. Mat., 2003, 41, 95-104.

\bibitem{Leonardi} S.Leonardi, On constants of some regularity theorems. De Giorgi type counterexample, Math. Nachr. 1998, 192, 191-204.

\bibitem{Leonardi-Leonetti-Pignotti} S.Leonardi, F.Leonetti, C.Pignotti, E.Rocha, V.Staicu, Maximum principles for some quasilinear elliptic systems, Nonlinear Anal. 2020, 194, 111377.

\bibitem{Leonardi-Leonetti-Pignotti2} S.Leonardi, F.Leonetti, C.Pignotti, E.Rocha, V.Staicu, Local boundedness for weak solutions to some quasilinear elliptic systems, Minimax Theory Appl.,  2021, 6, 365-372.

\bibitem{Leonetti-Rocha-Staicu} F.Leonetti, E.Rocha, V.Staicu, Quasilinear elliptic systems with measuredata, Nonlinear Anal., 2017, 154, 210-224.

\bibitem{Leonetti-Rocha-Staicu2} F. Leonetti, E. Rocha, V. Staicu, Smallness and cancellation in some elliptic systems with measure data, J. Math. Anal. Appl., 2018, 465, 885-902.

\bibitem{Manfredi} J.J.Manfredi, Monotone Sobolev functions. J. Geom. Anal., 1994, 4, 393-402.

\bibitem{Marcellini1}  P.Marcellini, Growth conditions and regularity for weak solutions to nonlinear elliptic pdes, J. Math. Anal. Appl., 2021, 501(1), 124408.

\bibitem{Marcellini2}  P.Marcellini, Regularity of minimizers of integrals in the calculus of variatioins with non standard growth conditions, Arch. Ration. Mech. Anal., 1989, 105, 267-284.

\bibitem{Marcellini3}  P.Marcellini, Regularity and existence of solutions of elliptic equations with $p,q$-growth conditions, J. Differ. Equ., 1991, 90, 1-30.

\bibitem{Marcellini4}  P.Marcellini, Regularity for elliptic equations with general growth conditions, J. Differ. Equ., 1993, 105, 296-333.

\bibitem{Marcellini5}  P.Marcellini, Everywhere regularity for a class of elliptic systems without growth conditions, Ann. Sc. Norm. Super. Pisa, Cl. Sci., 1996, 23, 1-25.

\bibitem{Marcellini6}  P.Marcellini, Regularity for some scalar variational problems under general growth conditions, J. Optim. Theory Appl., 1996, 90, 161-181.

\bibitem{Marcellini7}  P.Marcellini, Regularity under general and $p,q$-growth conditions, Discrete Contin. Dyn. Syst. Ser. S, 2020, 13, 2009-2031.

\bibitem{Meier} M. Meier, Boundedness and integrability properties of weak solutions of quasilinear elliptic systems, J. Reine Angew. Math. 333 (1982) 191-220.

\bibitem{Mingione} G. Mingione, Regularity of minima: an invitation to the dark side of the calculus of variations, Appl. Math. 51 (2006) 355- 426.

\bibitem{Mingione-Radulescu} G. Mingione, V. Radulescu, Recent developments in problems with nonstandard growth and nonuniform ellipticity, J. Math. Anal. Appl., 2021, 501, 125197.

\bibitem{Mooney-Savin} C. Mooney, O. Savin, Some singular minimizers in low dimensions in the calculus of variations, Arch. Rational Mech. Anal., 2016, 221, 1-22.

\bibitem{moser} J.Moser, A new proof of De Giorgi's theorem concerning the regularity problem for elliptic differential equations, Commun. Pure Appl. Math., 1980,  13, 457-468.

\bibitem{Nash} J.Nash, Continuity of solutions of parabolic and elliptic equations, Amer. J. Math., 1958, 80, 931-954.

\bibitem{PodioGuidugli} P.PodioGuidugli, De Giorgi counterexample in elasticity, Quart. Appl. Math. 1977, 34, 411-419.

\bibitem{SkrypnikandVoitovych} I.I.Skrypnik, M.V.Voitovych, ${\cal B}_1$ classes of De Giorgi-Lady\v{z}enskaya-Ural’tseva and their applications to elliptic and parabolic equations with generalized Orlicz growth conditions, Nonlinear Analysis, 2021, 202, 112135.

\bibitem{SkrypnikandVoitovych2} I.I.Skrypnik, M.V.Voitovych, On the generalized ${\cal U}_{m,p}^f$ classes of De Giorgi-Lady\v{z}enskaya-Ural’tseva and pointwise estimates of solutions to high-order elliptic equations via Wolff potentials, J. Diff. Eqns., 2020, 268, 6778-6820.

\bibitem{Stampacchia1} G.Stampacchia, Contributi alla regolarizzazione delle soluzioni dei problemi al contorno per equazioni del secondo ordine ellittiche, Ann. Sc. Norm. Sup. Pisa., 1958, 12, 223-245.

\bibitem{Stampacchia2}  G.Stampacchia, Problemi al contorno ellittici, con dati discontinui, dotati di soluzioni H\"olderiane, Ann. Mat. Pura Appl., 1960, 51, 1-37.

\bibitem{Stampacchia3}  G.Stampacchia, The space $L^{p,\lambda}$, $N^{p,\lambda}$ and interplation. Ann. Sc. Norm. Sup. Pisa., 1965, 25, 443-462.

\bibitem{Szekelyhidi} L.Sz\'ekelyhidi, The regularity of critical points of polyconvex functionals, Arch. Rational Mech. Anal., 2004, 172, 133-152.

\bibitem{Yan} Z.Q.Yan, Everywhere gularity for solutions to quasilinear elliptic systems of triangular form, Partial differential equations (Tianjin, 1986), Lecture Notes in Math., 1306, Springer, Berlin, 1988, 255-261.

\bibitem{Zacher} R.Zacher, A De Giorgi-Nash type theorem for time fractional diffusion equations, Math. Ann., 2013, 356, 99-146.












%






















%


















\end{thebibliography}
\end{document}